\documentclass[12pt,reqno]{amsart}
\usepackage{amsmath, amsthm, amssymb}
\usepackage{graphicx, ulem}
\usepackage{enumerate}
\usepackage{mathrsfs}
\usepackage{color,ulem}

\def\XXint#1#2#3{{\setbox0=\hbox{$#1{#2#3}{\int}$}
    \vcenter{\hbox{$#2#3$}}\kern-.5\wd0}}

\def\({\left(}
\def\){\right)}

\def\I{{\mathcal{I}_{\mu}}}

\def\RR{{\mathbb{R}}}

\newcommand{\eps}{\varepsilon}

\newcommand{\R}{\mathbb R}

\newcommand{\be}{\begin{equation}}
\newcommand{\ee}{\end{equation}}
\newcommand{\bea}{\begin{eqnarray}}
\newcommand{\eea}{\end{eqnarray}}
\newcommand{\beann}{\begin{eqnarray*}}
\newcommand{\eeann}{\end{eqnarray*}}

\def\e{\varepsilon}
\def\p{\partial}
\newcommand{\abs}[1]{\left\vert{#1}\right\vert}
\newcommand{\norm}[1]{\left\Vert{#1}\right\Vert}

\newtheorem{theorem}{Theorem}
\newtheorem{lemma}{Lemma}[section]

\newtheorem{remark}{Remark}[section]
\numberwithin{equation}{section}

\newenvironment{proof1}{
    \noindent {\em Proof }}{\hfill$\Box$}

\begin{document}

\title[Domain walls in the harmonic potential]{\bf Domain walls in the coupled Gross--Pitaevskii equations with the harmonic potential}

\author{Andres Contreras}
\address[A. Contreras]{Department of Mathematical Sciences, New Mexico State University, Las Cruces, New Mexico, USA}
\email{acontre@nmsu.edu}

\author{Dmitry E. Pelinovsky}
\address[D.E. Pelinovsky]{Department of Mathematics and Statistics, McMaster University,	Hamilton, Ontario, Canada, L8S 4K1}
\email{dmpeli@math.mcmaster.ca}

\author{Valeriy Slastikov}
\address[V. Slastikov]{School of Mathematics, University of Bristol, Bristol BS8 1UG, UK}
\email{valeriy.slastikov@bristol.ac.uk}

\date{\today}
\maketitle

\begin{abstract} 
	We study the existence and variational characterization of steady states in a coupled system of Gross--Pitaevskii equations modeling two-component Bose-Einstein condensates with the magnetic field trapping. The limit with no trapping has been the subject of recent works where domain walls have been constructed and several properties, including their orbital stability have been derived. Here we focus on the full model with the harmonic trapping potential and characterize minimizers according to the value of the coupling parameter $\gamma$. We first establish a rigorous connection between the two problems in the Thomas-Fermi limit via $\Gamma$-convergence. Then, we identify the ranges of $\gamma$ for which either the symmetric states $(\gamma < 1)$ or the uncoupled states $(\gamma > 1)$ are minimizers. Domain walls arise as minimizers in a subspace of the energy space with a certain symmetry for some $\gamma > 1$. We study bifurcation of the domain walls and furthermore give numerical illustrations of our results.
\end{abstract}

\section{Introduction}

Mixtures of Bose-Einstein condensates exhibit a wealth of phenomena studied extensively both in physics \cite{AoCh1998, Barankov2002,Han2012,NCK,PaBh2017, SaZuDa2011} and mathematics \cite{AfRo2020,AfRo2015,AfSa2020,GoRo2015,GoMe2017}. 
Topological defects such as domain walls and vortices appear in the segregation problems common in soft condensed matter   \cite{DrMaZe2011,Kasamatsu2013,LaKeTu2010-11,Malomed,Takeuch2011,Tripp2000}. Variational characterization of such defects is of great interest in mathematical physics \cite{AS,ABCP,Farina}. Related questions of their  stability in the time evolution play an important role in the understanding of the long time behavior of the underlying systems of dispersive equations \cite{CPP18,IaLe2014,NgWa2016}. 

Binary Bose-Einstein condensates trapped by a magnetic field are modeled by a coupled system of the Gross--Pitaevskii equations with a harmonic potential. A coupling parameter $\gamma$ determines the degree of repulsion between the two species. According to whether $\gamma$ is small or large, minimizers will exhibit very different qualitative behavior. In particular, three different types of steady states can be identified: \textit{uncoupled states}, \textit{symmetric states} and \textit{domain walls} (see Figure 11 in \cite{NCK}). The main goal of this work is to provide a variational characterization of these steady states. We study and determine regimes of existence and bifurcations of the domain walls in the presence of the harmonic confinement. We also study rigorously how these domain walls are related to those in the limiting problem without the trapping potential obtained in \cite{ABCP}. 

We consider the domain walls in the one-dimensional settings (see \cite{Malomed-review} for the review of domain walls in many physical settings).
Domain walls also arise in the two-dimensional and multi-dimensional geometry
where they are affected by the harmonic trapping and the boundary conditions \cite{DW-2D,DW-2D-Kevr}. Mathematical study of domain walls beyond the space 
of one dimension is opened for future work. 

The remainder of this section is organized as follows. We describe the coupled Gross--Pitaevskii equations and the associated energy with and without the harmonic potential in Section 1.1, where we also introduce the properties of the ground state of the scalar Gross--Pitaevskii equation and the Thomas--Fermi limit. Section 1.2 gives definitions of the uncoupled, symmetric, and domain wall states. Main results are described in Section 1.3. Numerical approximations of domain walls are reported in Section 1.4. Finally, Section 1.5 gives organization of the proofs of the main results.

\subsection{The Thomas-Fermi limit}

Let $\e>0$ be a small parameter for the semi-classical (Thomas--Fermi) limit. 
We consider a real-valued solution $\Psi = (\psi_1,\psi_2)$ to the following coupled system of Gross--Pitaevskii equations with the harmonic potential 
and the cubic nonlinear terms:
\be\label{DWHPVR}
\left.
\begin{matrix} -\e^2 \psi_1''(x) + x^2 \psi_1(x) + \(\psi_1(x)^2 + \gamma \psi_2(x)^2 - 1 \)\psi_1(x) = 0, \\
-\e^2 \psi_2''(x) + x^2 \psi_2(x) + \(  \gamma \psi_1(x) ^2 + \psi_2(x)^2 -1 \)\psi_2(x) = 0,
\end{matrix} \right\}\quad x\in\mathbb{R},
\ee
where $\gamma>0$ is the coupling parameter. The Thomas-Fermi limit corresponds to the asymptotic regime with $\e\to 0$.

System (\ref{DWHPVR}) is the Euler--Lagrange equation
for the following energy functional 
\begin{align}
\nonumber 
G_\e(\Psi) & = \frac{1}{2}\int_{\RR} \left[ \e^2 (\psi_1')^2 + \e^2 
(\psi_2')^2 + (x^2-1) (\psi_1^2 + \psi_2^2) \right.\\
& \qquad  + \frac{1}{2} (\psi_1^2 + \psi_2^2)^2 +
\left. (\gamma-1) \psi_1^2 \psi_2^2 \right] dx
\label{GPU}
\end{align}
for $\Psi = (\psi_1,\psi_2)$ defined in the energy space $\mathcal{E} := H^1(\RR;\mathbb{R}^2) \cap L^{2,1}(\RR;\mathbb{R}^2)$, where $L^{2,1}(\RR) := \{ f \in L^2(\RR) : \; |x| f \in L^2(\RR)\}$.

\begin{figure}[htpb!]
	\centering
	\includegraphics[width=8.5cm,height=6.5cm]{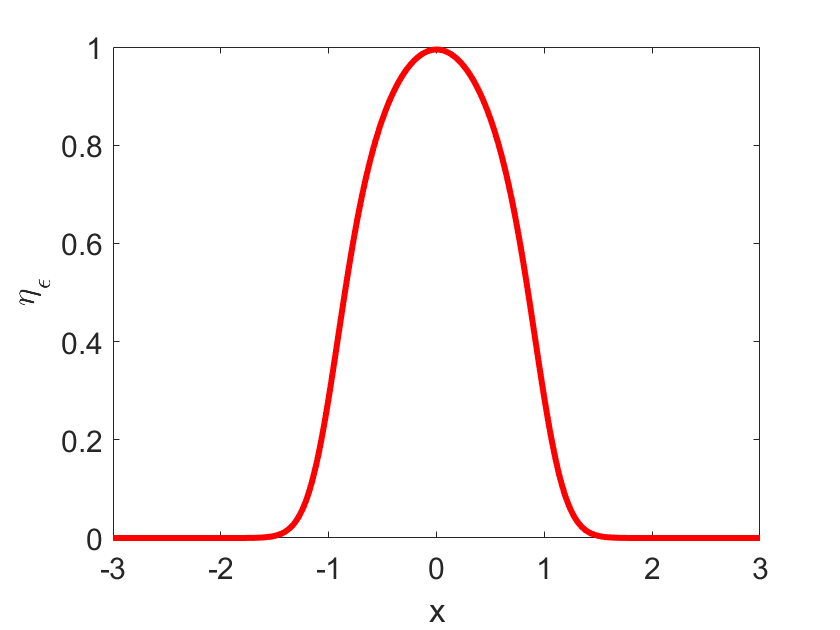}
	\caption{The profile of $\eta_{\e}$ versus $x$ for $\e = 0.1$. }
	\label{fig-1}
\end{figure}

In order to introduce the steady states of the coupled system of Gross--Pitaevskii equations, let us first describe the ground state of the scalar Gross--Pitaevskii equation with the harmonic potential.

Let $\eta_{\e} \in  H^1(\RR;\mathbb{R}) \cap L^{2,1}(\RR;\mathbb{R})$ be the positive solution of the differential equation
\be\label{T1}
\e^2\eta''_\e(x)+(1-x^2-\eta_\e^2(x))\eta_\e(x)=0,\quad x\in\mathbb{R}.
\ee
The positive solution shown on Fig. \ref{fig-1} has been studied 
in the literature in many details \cite{GaPe,IgMi1, IgMi2}. 
The formal limit $\e \to 0$ corresponds to the so-called Thomas-Fermi cloud 
$\eta_0$ of the form
\begin{equation}
\label{TF-limit}
\eta_0(x) :=\sqrt{1-x^2}\;{\bf 1}_{\{\abs{x}<1\}}.
\end{equation}
By Proposition 2.1 in \cite{IgMi1}, the positive solution $\eta_{\e}$ satisfies
\begin{equation}
\label{P-1}
0 < \eta_\e(x) \leq 1, \quad x \in \RR
\end{equation}
and 
\begin{equation}
\label{P-1-plus}
\eta_{\e}(x) \leq \eta_0(x), \quad x \in [-1+\eps^{1/3},1-\eps^{1/3}],
\end{equation}
moreover, $\eta_{\e}$ converges to $\eta_0$ in the sense that for every compact subset $K$ of $\{\abs{x}<1\}$, there exists a constant $C_K > 0$ such that
\be
\label{P-2}
\norm{\eta_\e-\eta_0}_{C^1(K)}\leq C_K \e^2.
\ee
Refined asymptotic properties of $\eta_\e$ were further obtained in \cite{GaPe}, where existence of $\eta_{\e}$ was obtained from 
the Hasting--McLeod solution of the Painlev\'e-II equation. 
By Theorem 1 in \cite{GaPe}, there exists a constant $C > 0$ such that
\be
\label{P-3}
\norm{\eta_\e-\eta_0}_{\infty} \leq C \e^{\frac{1}{3}}, \quad 
\norm{\eta'_\e}_{\infty} \leq C \e^{-\frac{1}{3}}.
\ee

Equation (\ref{T1}) is the Euler--Lagrange equation for the following energy functional 
\begin{equation}
\label{energy-F}
F_\e (\eta) = \frac12 \int_\R \left[ \e^2 (\eta')^2 + (x^2-1) \eta^2 + \frac12 \eta^4  \right] \, dx.
\end{equation}
By using the transformation
\begin{equation}
\label{transformation}
\psi_{1,2}(x) = \eta_{\e}(x) \phi_{1,2}\(\frac{x}{\e}\)
\end{equation}
and making the change of variables $x \to z := x/\e$, 
we obtain $G_{\e}(\Psi) = F_{\e}(\eta_\e) + \e J_{\e}(\Phi)$,  
where
\begin{align}
J_\e(\Phi) &= \frac12 \int_\R \left[
\eta_\e(\e \cdot)^2  (\phi_1')^2 + \eta_\e(\e \cdot)^2  (\phi_2')^2
+ \frac12 \eta_\e(\e \cdot)^4 (\phi_1^2 + \phi_2^2 -1)^2 \right. 
\nonumber \\
& \qquad  + (\gamma-1) \eta_\e(\e \cdot)^4 \phi_{1}^2 \phi_{2}^2 \bigg] \,dz
\label{GPU-reduced}
\end{align}
for $\Phi = (\phi_1,\phi_2)$ defined in $H^1_{\rm loc}(\RR;\RR^2)$ for which $J_\e(\Phi) < \infty$. 

Since $\lim\limits_{\e \to 0} \eta_{\e}(\e z) = \eta_0(0) = 1$ for every $z \in \RR$, the formal limit of (\ref{GPU-reduced}) is the energy functional 
studied in \cite{ABCP}:
\begin{align}
J_0(\Phi) = \frac12 \int_\R \left[  (\phi_1')^2 + (\phi_2')^2 + \frac12 (\phi_1^2 + \phi_2^2 -1)^2 + (\gamma-1)\phi_1^2 \phi_2^2 \right] \, dz.
\label{GPU-0}
\end{align}
The Euler--Lagrange equations for $J_0(\Phi)$ is the coupled system 
of homogeneous Gross--Pitaevskii equations without the harmonic potential:
\begin{equation}\label{hc}
\left.\begin{aligned}
- \phi''_1(z) + \left( \phi_1(z)^2 + \gamma \phi_2(z)^2 -1\right)\phi_1(z) &=0, \\
- \phi''_2(z) + \left( \gamma \phi_1(z)^2 + \phi_2(z)^2 -1\right)\phi_2(z) & =0,
\end{aligned}\right\} \quad z \in\RR.
\end{equation}
A symmetric pair of domain wall solutions exists for $\gamma > 1$ 
as minimizers of $J_0$ by Theorem 1.1 in \cite{ABCP} (see \cite{Berg} for earlier mathematical results and \cite{Mineev} for the first fundamental result 
in physics literature). One particular domain wall solution $\Phi = (\phi_1,\phi_2)$ to the system (\ref{hc}) 
satisfies the following boundary conditions:
\begin{gather}
\label{plus}
\left. \begin{aligned}
\phi_1(z)\to 0, \quad \phi_2(z)\to 1, \qquad\text{as $z\to-\infty$,}\\
\phi_1(z)\to 1, \quad \phi_2(z)\to 0, \qquad\text{as $z\to +\infty$,}
\end{aligned}\right\}
\end{gather}
whereas the other domain wall can be obtained by the transformation 
$\phi_1 \leftrightarrow \phi_2$.

Theorems 2.1, 2.4, and 3.1 in \cite{ABCP} state that $(\phi_1,\phi_2)$ 
for the domain wall solution with the boundary conditions (\ref{plus})  satisfies the following properties:
\begin{itemize}
	\item[(a)] $\phi_2(z)=\phi_1(-z)$ for all $z\in \mathbb{R}$;
	\item[(b)] $\phi_1^2(z) + \phi_2^2(z) \leq 1$ for all $z \in \mathbb{R}$;
	\item[(c)] $\phi_1'(z) > 0$ and $\phi_2'(z) < 0$ for all $z \in \mathbb{R}$.
\end{itemize}
Uniqueness of domain walls with properties (a)--(c) was shown in \cite{AS} and more generally in \cite{Farina}.
Orbital stability of domain walls in the coupled system of Gross--Pitaevskii equations was established in a weighted Hilbert space in
\cite{CPP18}.

\subsection{Uncoupled, Symmetric and Domain Wall States}

We introduce three steady state solutions to system (\ref{DWHPVR}) that will be studied in this work. 

\begin{enumerate}
	\item[(S1)] For every $\gamma >0$, there exist {\em uncoupled states}
\begin{equation}
\label{uncoupled-state}
\psi_1(x) = \eta_{\e}(x), \quad \psi_2(x) = 0, \quad \mbox{\rm and} \quad
\psi_1(x) = 0, \quad \psi_2(x) = \eta_{\e}(x).
\end{equation}

	\item[(S2)] For every $\gamma > 0$, there exists the {\em symmetric state}
\be
\label{positive-state}
\psi_1(x) = \psi_2(x) = (1 + \gamma)^{-1/2} \eta_{\e}(x).
\ee

	\item[(S3)] We say that the solution is the {\em domain wall state} if $\psi_1(x) = \psi_2(-x)$ for all $x \in \mathbb{R}$ but $\psi_1(x) \not\equiv \psi_2(x)$.
\end{enumerate}

\begin{figure}[htpb!]
	\centering
	\includegraphics[width=9cm,height=7cm]{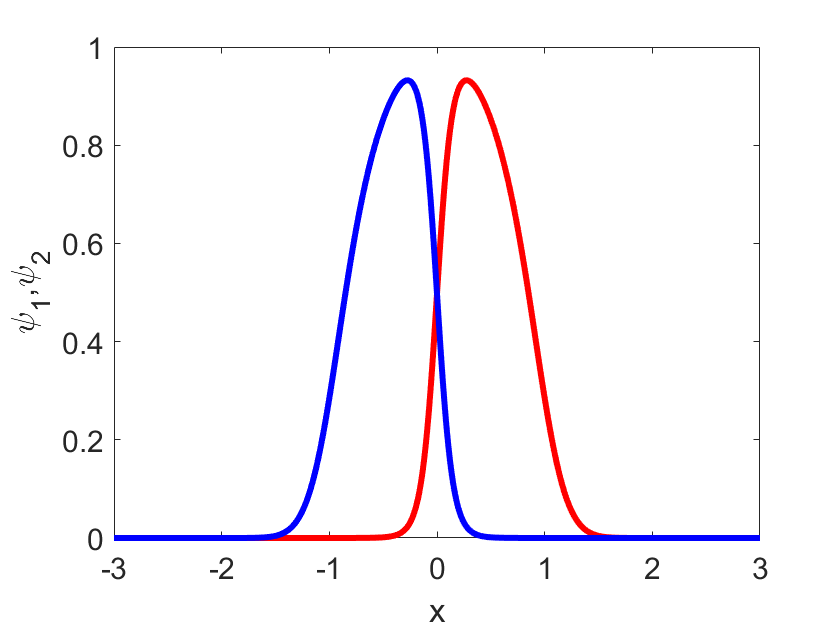}
	\caption{The profile of the domain wall state versus $x$ for $\e = 0.1$ and $\gamma = 3$. }
	\label{fig-2}
\end{figure}

Both the symmetric state (S2) and the domain wall state (S3) are defined in the energy space with the symmetry given by 
\begin{equation}
\label{energy-space}
\mathcal{E}_{s} := \{ \Psi \in \mathcal{E} \, : \quad \psi_1(x) = \psi_2(-x), \;\; x \in \RR \}.
\end{equation}
One domain wall state (S3) with the property $\psi_1(x) > \psi_2(x) > 0$ for $x > 0$ is shown on Fig. \ref{fig-2}. Another domain wall state can be obtained 
by the transformation $\psi_1 \leftrightarrow \psi_2$.
Existence, variational characterization, and bifurcations of 
the domain wall state (S3) are the main subjects of this work.

\subsection{Main results}

Our first theorem establishes a rigorous connection between problems \eqref{DWHPVR} and \eqref{hc}. Recall the product form (\ref{transformation}) 
which reduces the minimization problem for $G_\e$ to one for $J_\e$. A consequence of Lemma \ref{sepen} below is that $\Psi \in \mathcal{E}$ minimizes $G_\e$ if and only if $\Phi \in H^1_{\rm loc}(\RR;\RR^2)$ minimizes $J_\e$. Furthermore, we prove that $J_\e$ converges to $J_0$ in the sense of $\Gamma$-convergence \cite{dalmaso}.

\begin{theorem}\label{gconv1}
	Let $\gamma>1$, $\e>0$ and $\Phi_\e$ be a sequence in $H^1_{\rm loc}(\R,\mathbb{R}^2)$. Then, 
	\begin{itemize}
		\item[(i)] If $J_\e(\Phi_\e) < \infty$ uniformly in $\e$ then as $\e \to 0$ we have a subsequence (not relabelled) $\Phi_\e \rightharpoonup \Phi$ in $H^1_{\rm loc}(\R,\mathbb{R}^2)$.
		\item[(ii)] \emph{($\Gamma-\liminf$ inequality)} If $\Phi_\e \to \Phi$ in $L^2_{\rm loc}(\R,\mathbb{R}^2)$, then 
		\begin{equation}
		\label{inequality-1}
		\liminf_{\e \to 0} J_\e(\Phi_\e) \geq J_0(\Phi).
		\end{equation}
		\item[(iii)] \emph{($\Gamma-\limsup$ inequality)} If $\Phi \in H^1_{\rm loc}(\R,\mathbb{R}^2)$, then there is a sequence 
		
\noindent $\Phi_\e \in H^1_{\rm loc}(\R,\mathbb{R}^2)$ such that as $\e \to 0$ we have $\Phi_\e \to \Phi$ in $L^2_{\rm loc}(\R,\mathbb{R}^2)$ and 
		\begin{equation}
		\label{inequality-2}
		\limsup_{\e \to 0} J_\e(\Phi_\e) \leq J_0(\Phi).
		\end{equation}
	\end{itemize}
\end{theorem}	
In view of the decomposition $G_{\e}(\Psi) = F_{\e}(\eta_\e) + \e J_{\e}(\Phi)$,
Theorem \ref{gconv1} establishes a relation between minimizers of $G_\e$ for small $\e > 0$ and those of the limiting functional $J_0$. This relation was one of the motivations for  studying domain walls in \cite{ABCP, CPP18} as minimizers of $J_0$ for $\gamma>1$ with prescribed limits at infinity.

The functional $J_0$ also possesses critical points $(1,0)$, $(0,1)$, and $(1+\gamma)^{-1/2} (1,1)$, which correspond to the steady states (S1) and (S2) in the limit $\e \to 0$. The question of their variational characterization for $\gamma > 1$ was left open in \cite{ABCP}. 
The case $0 <\gamma<1$ was not considered in \cite{ABCP}. The following two theorems clarify the variational characterization of the three steady states (S1), (S2), and (S3) as critical points of the energy $G_{\e}$.

\begin{theorem}
For every $\e > 0$, the symmetric state (S2) is a global minimizer of the energy $G_{\e}$ in the energy space $\mathcal{E}$ if $\gamma \in (0,1)$ and a saddle point of the energy if $\gamma \in (1,\infty)$. Moreover, the symmetric state is the only positive solution of the coupled system \eqref{DWHPVR} if $\gamma \in (0,1)$. The uncoupled states (S1) are saddle points of the energy if $\gamma \in (0,1)$ and the only global minimizers of the energy if $\gamma \in (1,\infty)$.
\label{theorem-variational}
\end{theorem}

\begin{theorem} 
There exists $\e_0 > 0$ such that for every $\e \in (0,\e_0)$ there is $\gamma_0(\e) \in (1,\infty)$ such that the symmetric state (S2) is a global minimizer of the energy $G_{\e}$ in the energy space with symmetry $\mathcal{E}_s$ if
$\gamma \in (0,\gamma_0(\e)]$ and a saddle point of the energy in $\mathcal{E}_s$ if $\gamma \in (\gamma_0(\e),\infty)$, where $\gamma_0(\e) \to 1$ as $\e \to 0$. Two domain wall states (S3) are global minimizers of the energy $G_{\e}$ in $\mathcal{E}_s$ if $\gamma \in (\gamma_0(\e),\infty)$: one satisfies $\psi_1(x) > \psi_2(x) > 0$ for $x > 0$ and the other one obtained by the transformation $\psi_1 \leftrightarrow \psi_2$.
\label{theorem-domain-walls}
\end{theorem}

\begin{figure}[htpb!]
	\centering
	\includegraphics[width=8.5cm,height=6.5cm]{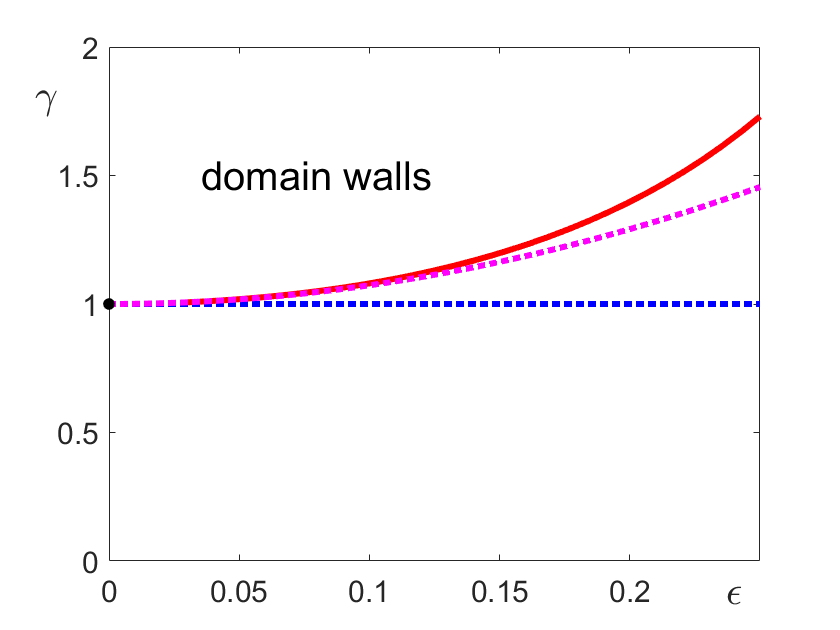}
	\caption{Bifurcation diagram for the domain wall solutions on the 
		$(\e,\gamma)$ plane. Domain walls exist above the bifurcation curve $\gamma = \gamma_0(\e)$. }
	\label{fig-bif}
\end{figure}

The bifurcation curve $\gamma = \gamma_0(\e)$ is shown on Fig. \ref{fig-bif} 
on the $(\e,\gamma)$ plane. Domain wall states (S3) are global minimizers of the energy $G_{\e}$ in $\mathcal{E}_s$ above the bifurcation curve. 
The blue dotted curve shows the straight line $\gamma = 1$. 
The magenta dotted curve shows the approximation of the bifurcation curve 
explained below. 

It follows from the exponential decay of the domain wall states of
	the homogeneous system (\ref{hc}) to the boundary conditions (\ref{plus}) that the domain wall states without the harmonic potential
only exist for $\gamma > 1$, 
where they are minimizers of the energy $J_0$ \cite{ABCP}. 
It is surprising that we are not able to derive a similar result 
for the domain wall states (S3) in the case of the harmonic potential. By uniqueness of positive solutions 
of the coupled system (\ref{DWHPVR}) in Theorem \ref{theorem-variational}, the domain wall do not exist for $\gamma \in (0,1)$, but we do not have a similar non-existence result for 
$\gamma \in [1,\gamma_0(\e)]$. By Theorem \ref{theorem-domain-walls}, the domain wall states exist and are minimizers of the energy $G_{\e}$ if $\gamma \in (\gamma_0(\e),\infty)$.

In order to understand more 
properties in the construction of the domain wall states (S3), 
we introduce the Dirichlet condition 
$\psi_1(0) = \psi_2(0) = \alpha$ and define the energy space on the half-line:
\begin{equation}
\label{energy-half-line}
\mathcal{D}_{\alpha} := \{ \Psi \in H^1(\R_+; \R^2) \cap L^{2,1}(\R_+;\R^2) \, : \quad \psi_1(0) = \psi_2(0) = \alpha \}.
\end{equation}

The following theorem states that if $\alpha > 0$ is fixed and the non-symmetric critical points of $G_{\e}$ exist in $D_{\alpha}$, they are minimizers of $G_{\e}$ in $D_{\alpha}$. Unfortunately, this result does not eliminate the domain wall states in $\mathcal{E}_s$ for $\gamma \in [1,\gamma_0(\e)]$ since such states could be saddle points of $G_{\e}$ with respect to varying $\alpha$.

\begin{theorem}
	\label{DWalpha}
	Fix $\e > 0$, $\gamma > 0$, and $\alpha > 0$. Let $\Psi \in \mathcal{D}_{\alpha}$ be a critical point of the energy $G_{\e}$ satisfying $\psi_1 (x) > \psi_2(x) > 0$ for all $x \in (0,\infty)$. Then $\Psi = (\psi_1,\psi_2)$ and $\Psi' = (\psi_2,\psi_1)$ are the only global minimizers of the energy $G_{\e}$ in $\mathcal{D}_{\alpha}$.
\end{theorem} 

Next, we study bifurcations of the steady states. The exceptional case $\gamma = 1$ in Theorem \ref{theorem-variational} is special since the nonlinear system (\ref{DWHPVR}) with $\gamma = 1$ has an additional 
rotational symmetry which results in the existence of the following one-parameter family of solutions:
\begin{equation}
\label{rotating-state}
\psi_1(x) = \cos \theta \; \eta_{\e}(x), \quad \psi_2(x) = \sin \theta \; \eta_{\e}(x),
\end{equation}
where $\theta \in [0,\frac{\pi}{2}]$ is required for positivity of solutions. The family of solutions (\ref{rotating-state})  recovers 
(S1) if $\theta = 0$ or $\theta = \frac{\pi}{2}$ and (S2) if $\theta = \frac{\pi}{4}$. The following theorem suggests that the domain wall state (S3)
does not bifurcate from the one-parameter family (\ref{rotating-state}), 
in agreement with Theorem \ref{theorem-domain-walls}.

\begin{theorem}
For small $\e > 0$, the one-parameter family
(\ref{rotating-state}) for $\gamma = 1$ is only continued for
$\gamma \neq 1$ either as the uncoupled states (S1) or as the
symmetric state (S2). 
\label{theorem-bifurcation}
\end{theorem}

When $\gamma=1$ and $\e = 0$ (black dot on Fig. \ref{fig-bif}), the energy $J_0$ ceases to impose segregation of the two  components and hence, a natural question is to identify the valid asymptotic behavior of the domain walls in the limit $(\e,\gamma) \to (0,1)$. 
We do so by using the rescaling  
$$
z \mapsto y := z \sqrt{\gamma-1}, \quad  
\Phi(z) = \Theta(y), \quad \e = \mu \sqrt{\gamma - 1}, \quad 
\mbox{\rm and} \quad J_{\e}(\Phi) = \sqrt{\gamma-1} I_{\mu,\gamma}(\Theta),
$$
where
\begin{eqnarray}
\nonumber
I_{\mu,\gamma} (\Theta) &=&  \frac{1}{2} \int_{\mathbb{R}} \left[ \eta_{\e}(\mu \cdot)^2 (\theta_1')^2 + \eta_{\e}(\mu \cdot)^2 (\theta_2')^2
+  \frac{1}{2(\gamma-1)}  \eta_{\e}(\mu \cdot)^4 (\theta_1^2 + \theta_2^2 - 1)^2 \right. \\ 
\label{energy-gamma}
&&  \qquad +  \eta_{\e}(\mu \cdot)^4 \theta_1^2 \theta_2^2 \bigg] \, dy.
\end{eqnarray}
Since $\e = \mu \sqrt{\gamma - 1}$ and $\mu > 0$ is fixed, the formal limit of 
(\ref{energy-gamma}) as $\gamma \to 1^+$ is 
\begin{align}
I_{\mu,1}(\Theta) = 
\begin{cases} 
\frac{1}{2} \int_{\I} \left[ \eta_{0}(\mu \cdot)^2 (\theta_1')^2 + \eta_{0}(\mu \cdot)^2 (\theta_2')^2 + \eta_{0}(\mu \cdot)^4 \theta_1^2 \theta_2^2 \right]  \, dy \; &\hbox{ if } \theta_1^2 + \theta_2^2 =1, \\
+\infty \; &\hbox{ otherwise, }
\end{cases}
\label{energy-gamma-reduced}
\end{align}
where $\I :=(-\frac{1}{\mu}, \frac{1}{\mu})$.

Minimizers of $I_{\mu,1}$ in (\ref{energy-gamma-reduced}) 
are shown on Fig. \ref{fig-3} by dashed line 
together with the numerical approximation of the domain wall 
state shown by the solid line for $\e = 0.1$ and $\gamma = 1.2$. 
Both solutions are shown versus the original coordinate $x$. The agreement implies the relevance of the asymptotic approximation found from $I_{\mu,1}$.

\begin{figure}[htpb!]
	\centering
	\includegraphics[width=8.5cm,height=6.5cm]{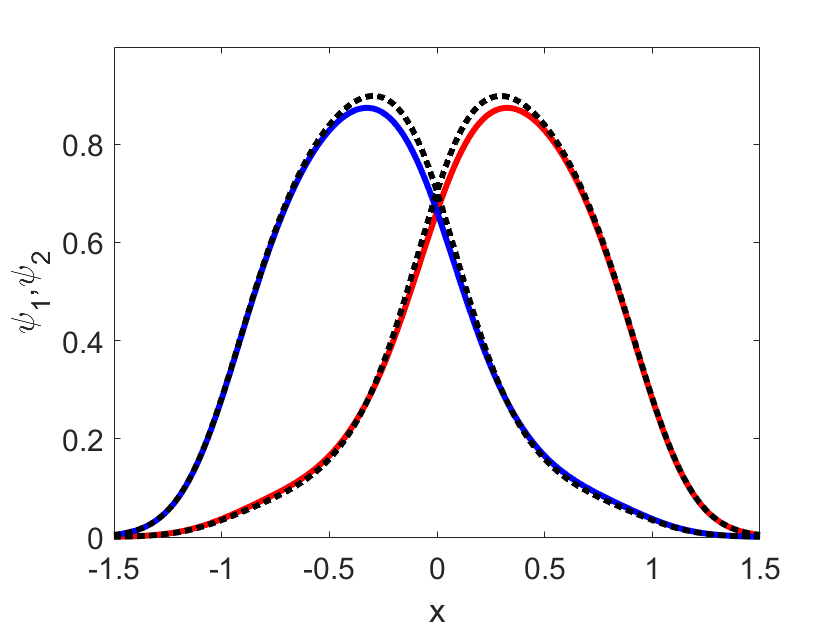}
	\caption{The profile of the domain wall state (solid) versus $x$ 
		and the asymptotic approximation obtained from the minimizers of the energy functional $I_{\mu,1}$ for $\e = 0.1$ and $\gamma = 1.2$.}
	\label{fig-3}
\end{figure}

A rigorous connection between the two energy functionals $I_{\mu,\gamma}$ and $I_{\mu,1}$ given by (\ref{energy-gamma}) and (\ref{energy-gamma-reduced}) 
is established in Lemma \ref{theorem-gamma} below, where $\Gamma$ 
convergence of $I_{\mu,\gamma}$ to $I_{\mu,1}$ is shown for fixed $\mu > 0$.

The following theorem states the existence of a special $\mu_0 \in (0,\infty)$ 
separating the symmetric state (S2) and the domain wall state (S3) 
as minimizers of $I_{\mu,1}$ in the energy space with symmetry $\mathcal{E}_s$. 
The curve $\e = \mu_0 \sqrt{\gamma - 1}$ is shown on Fig. \ref{fig-bif} by 
the magenta dotted line, illustrating a good approximation for the 
bifurcation curve $\gamma = \gamma_0(\e)$.

\begin{theorem}
	\label{theorem-asymptotic}
	There exists $\mu_0 \in (0,\infty)$ such that the symmetric state 
	(S2) is a global minimizer of the energy $I_{\mu,1}$ in $\mathcal{E}_s$ if  $\mu \in [\mu_0,\infty)$ and a saddle point of the energy in $\mathcal{E}_s$ if $\mu \in (0,\mu_0)$. The domain wall states (S3) exist only if $\mu \in (0,\mu_0)$, where they are global minimizers of the energy $I_{\mu,1}$ in $\mathcal{E}_s$.
\end{theorem}

If we wish to set $\mu = 0$ in $I_{\mu,1}$, we obtain the limiting energy 
for the energy functional $J_{0}$ in the sense of $\Gamma$-convergence:
\begin{align}
I_{0,1}(\Theta) = 
\begin{cases} 
\frac{1}{2} \int_{\mathbb{R}} \left[ (\theta_1')^2 + (\theta_2')^2 + \theta_1^2 \theta_2^2 \right]  \, dy \quad &\hbox{ if } \theta_1^2 + \theta_2^2 =1, \\
+\infty \quad &\hbox{ otherwise. }
\end{cases}
\label{energy-gamma-limiting}
\end{align}
The domain wall state (S3) can be found explicitly as a minimizer of $I_{0,1}$ 
since the symmetric state (S2) corresponding to $\theta_1 = \theta_2 = \frac{1}{\sqrt{2}}$ is a saddle point by Theorem \ref{theorem-asymptotic}. Using the representation $(\theta_1,\theta_2) = (\sin(u),\cos(u))$
on the circle $\theta_1^2 + \theta_2^2 = 1$ yields $I_{0,1}(\Theta)$ 
rewritten as $I_{0,1}(u)$ in the form:
$$
I_{0,1}(u) = \frac{1}{2} \int_{\RR} \left[ (u')^2 + \frac{1}{4} \sin^2(2u) \right] dy,
$$
with the corresponding Euler--Lagrange equation given by 
\begin{equation}
\label{EL}
-u'' + \frac{1}{4} \sin(4u) = 0.
\end{equation}
The domain wall state in $\mathcal{E}_s$ corresponds to the restriction $\theta_1(y) = \theta_2(-y)$ for $y \in \mathbb{R}$. 
In particular, $\theta_1(0) = \theta_2(0) = \frac{1}{\sqrt{2}}$ which yields the condition $u(0) = \frac{\pi}{4}$. 
The boundary conditions $\theta_1(y) \to 0$ as $y \to -\infty$ 
and $\theta_1(y) \to 1$ as $y \to +\infty$ yields 
the conditions $u(y) \to 0$ as $y \to -\infty$ and $u(y) \to \frac{\pi}{2}$ as $y \to +\infty$. The unique solution of the differential equation 
(\ref{EL}) satisfying these requirements is found in the form:
\begin{equation}
\label{EL-solution}
u(y) = \frac{\pi}{2} - \arctan(e^{-y}), \quad y \in \RR.
\end{equation}
This solution is the asymptotic approximation as $\gamma \to 1$
of the domain walls of the coupled system (\ref{hc}) without the harmonic potential. This explicit approximation was not obtained
in \cite{ABCP}.

\subsection{Numerical approximations of domain walls}

By Theorem \ref{theorem-variational}, the domain wall states (S3) are not global minimizers of $G_\e$ in $\mathcal{E}$ in their range of existence. In \cite{ABCP}, domain walls were obtained variationally as heteroclinics to specific limits at infinity. Another approach to constructing domain walls is to consider a problem in the energy space with a spatial symmetry 
denoted by $\mathcal{E}_s$ and defined in (\ref{energy-space}).
By Theorem \ref{theorem-domain-walls}, 
the domain wall states (S3) are global minimizers of $G_{\e}$ in $\mathcal{E}_s$ for $\gamma \in (\gamma_0(\e),\infty)$.

The effective numerical method for computation of domain walls 
is developed by working in the energy space $\mathcal{D}_{\alpha}$ 
on the half-line defined in (\ref{energy-half-line}) subject to the 
Dirichlet condition $\psi_1(0) = \psi_2(0) = \alpha$ with $\alpha > 0$.
The symmetric state in $\mathcal{D}_{\alpha}$ corresponds to the reduction $\psi_1(x) = \psi_2(x)$ for $x > 0$. In addition, 
there may exist two non-symmetric states in $\mathcal{D}_{\alpha}$: one satisfies $\psi_1(x) > \psi_2(x) > 0$ for $x  > 0$ and the other one is obtained by the transformation $\psi_1 \leftrightarrow \psi_2$. By Theorem \ref{DWalpha}, if the non-symmetric states exist, they are global minimizers of the energy $G_{\e}$ in $\mathcal{D}_{\alpha}$.

Non-symmetric minimizers of $G_{\e}$ in $\mathcal{D}_{\alpha}$ depend on $\alpha$ and recover the domain wall states (S3) on the half-line if 
the split function $S_{\e}(\alpha)$ vanishes, where 
$$
S_{\e}(\alpha) := \psi_1'(0^+) + \psi_2'(0^+)
$$
and the derivatives are well-defined because minimizers of $G_{\e}$ in $\mathcal{D}_{\alpha}$ are smooth on $\mathbb{R}_+$. Another criterion to identify the domain wall state (S3) as the global minimizers of $G_{\e}$ in $\mathcal{E}_s$ is to plot the energy level 
$G_{\e}(\Psi_{\alpha})$  versus $\alpha$, where $\Psi_{\alpha}$ is the  non-symmetric minimizer of $G_{\e}$ in $\mathcal{D}_{\alpha}$, and to find the values of $\alpha$ for which $G_{\e}(\Psi_{\alpha})$ is minimal. 

We have implemented both ways to find the domain wall states (S3) numerically. Non-symmetric minimizers of $G_{\e}$ in $\mathcal{D}_{\alpha}$ were approximated numerically with a second-order central-difference relaxation method. 
Fig. \ref{fig-4} shows how $S_{\e}(\alpha)$ and $G_{\e}(\Psi_{\alpha})$ 
depend on $\alpha$ for $\e = 0.1$ and $\gamma = 3$. Zero of $S_{\e}$ 
corresponds to the minimum of $G_{\e}(\Psi_{\alpha})$ and appears 
at $\alpha_0 \approx 0.5$. The corresponding domain wall state 
extended from $\mathcal{D}_{\alpha_0}$ 
by the symmetry condition in space $\mathcal{E}_s$ is shown on Fig. \ref{fig-2}. 

\begin{figure}[htpb!]
	\centering
	\includegraphics[width=7cm,height=5cm]{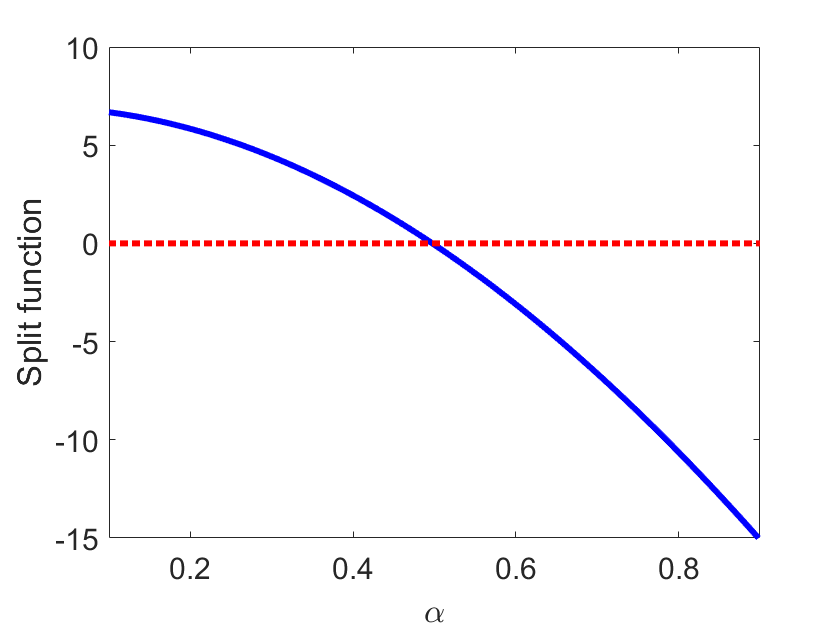}
	\includegraphics[width=7cm,height=5cm]{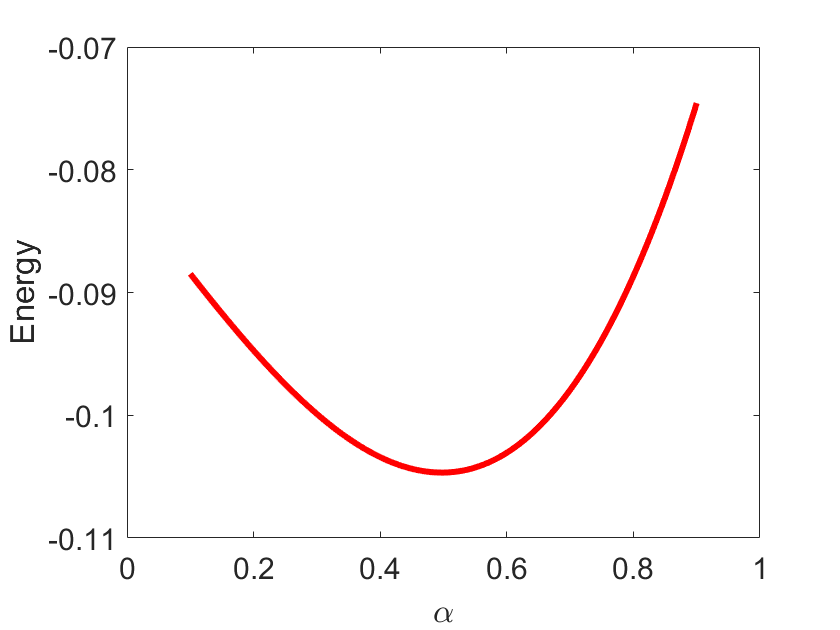}
	\caption{Plots of $S_{\e}(\alpha)$ (left) and $G_{\e}(\Psi_{\alpha})$ (right) versus $\alpha$ for $\e = 0.1$ and $\gamma = 3$. }
	\label{fig-4}
\end{figure} 

Performing numerical approximations in the range of values of $\gamma$, 
we can obtain a family of domain wall states by using this method. Fig. \ref{fig-5} shows the dependence of the optimal value of $\alpha = \psi_1(0) = \psi_2(0)$ found from the root of $S_{\e}$ versus $\gamma$ by solid line. Iterations of the numerical method do not converge near $\gamma = 1$ 
since the domain wall states are minimizers of $G_{\e}$ in $\mathcal{E}_s$ only for $\gamma \in (\gamma_0(\e),\infty)$ by Theorem \ref{theorem-domain-walls} with $\gamma_0(\e) > 1$ for $\e > 0$. The dashed line shows the linear interpolation between the limiting value $\alpha_0 = \frac{1}{\sqrt{2}}$ found from the limiting energy (\ref{energy-gamma-reduced})
shown by the blue dot and the last numerical data point at $\gamma = 1.2$. 

\begin{figure}[htpb!]
	\centering
	\includegraphics[width=8cm,height=6cm]{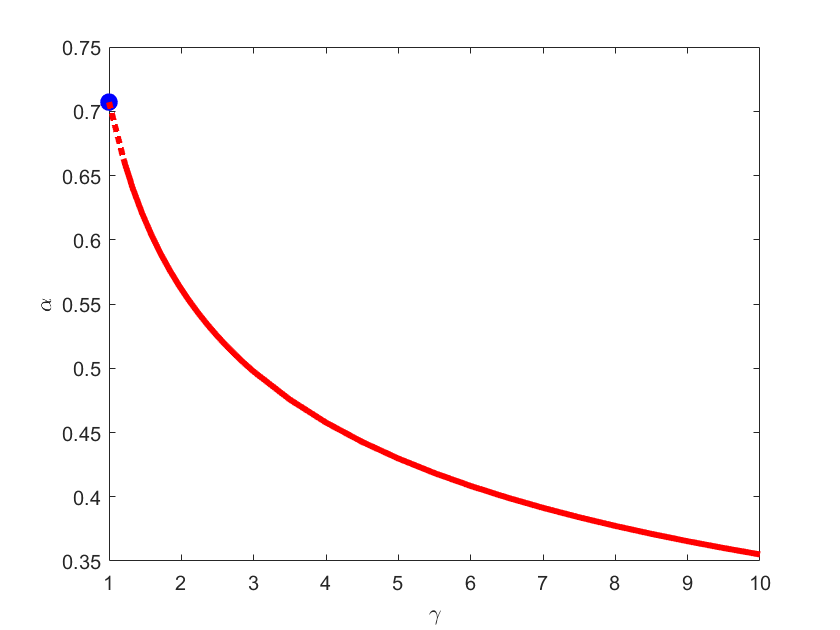}
	\caption{Dependence of the optimal value of $\alpha$ versus $\gamma$ for $\e = 0.1$. The blue dot shows the limiting value $\alpha_0 = \frac{1}{\sqrt{2}}$ as $\gamma \to 1^+$ found from the limiting energy (\ref{energy-gamma-reduced}). }
	\label{fig-5}
\end{figure}

Fig. \ref{fig-3} (solid line) shows the domain wall state at 
$\e = 0.1$ and $\gamma = 1.2$. 
In order to find the minimizer of $I_{\mu,1}$ in (\ref{energy-gamma-reduced}), we write $(\theta_1,\theta_2) = (\sin(u),\cos(u))$ and solve the corresponding Euler--Lagrange equation 
\begin{equation}
\label{EL-numerics}
- \frac{d}{dy} \left[ \eta_0(\mu y)^2 \frac{du}{dy} \right] + \frac{1}{4} \eta_0(\mu y)^4 \sin(4u) = 0, \quad y \in (\mu^{-1},\mu^{-1}).
\end{equation}
The Euler--Lagrange equation (\ref{EL-numerics}) is solved numerically 
subject to the conditions $u(0) = \frac{\pi}{4}$ and $u'(y) > 0$ for $y \in (-\mu^{-1},\mu^{-1})$. The solution is concatenated with the constant solution $u(y) = u(\mu^{-1})$ for $y \in (\mu^{-1},\infty)$ 
and $u(y) = u(-\mu^{-1})$ for $y \in (-\infty,-\mu^{-1})$.
The second-order central-difference relaxation method is used to obtain the numerical approximation of $u(y)$, which is then plotted versus the original coordinate $x$.

Fig. \ref{fig-1} is obtained by using the same relaxation method for the stationary equation (\ref{T1}) on $\mathbb{R}_+$ subject to the conditions 
$\eta_{\e}'(0) = 0$ and $\eta_{\e}(x) \to 0$ as $x \to \infty$.

The bifurcation curve $\gamma= \gamma_0(\e)$ on Fig. \ref{fig-bif} is obtained by 
numerical approximation of eigenvalues of the Hessian operator at the symmetric 
state (S2) in the proof of Theorem \ref{theorem-domain-walls}. Eigenvalues of the Dirichlet problem for the Schr\"{o}dinger operator 
\begin{equation}
\label{Schr-operator}
L_{\gamma} = -\e^2 \partial_x^2 + x^2 - 1 + \eta_{\e}^2 + 2 \frac{1-\gamma}{1+\gamma} \eta_{\e}^2, \quad x \in \mathbb{R}_+
\end{equation}
are shown on Fig. \ref{fig-eig}. They are monotonically decreasing in $\gamma$ 
and the crossing point of the first (smallest) eigenvalue of $L_{\gamma}$ gives the bifurcation curve $\gamma = \gamma_0(\e)$. The same second-order central-difference method is applied to approximate derivatives of the Schr\"{o}dinger operator $L_{\gamma}$ on $\mathbb{R}_+$ subject to the Dirichlet condition at $x = 0$. The spatial domain is truncated on $[0,3]$ similar to Fig. \ref{fig-1}.

The magenta dotted curve on Fig. \ref{fig-bif} is obtained from the first (smallest) eigenvalue $\nu_0 := \mu_0^{-2}$ of the boundary-value problem 
\begin{equation}
\label{limit-Schr-op}
\left. \begin{array}{l} 
- \frac{d}{dx} \left[ (1-x^2) \frac{dv}{dx} \right] = \nu (1-x^2)^2 v, \quad 0 < x < 1, \\
v(0) = 0, \quad v'(1) = 0, \end{array} \right\} 
\end{equation}
arising in the proof of Theorem \ref{theorem-asymptotic}. 
With the transformation $x = \tanh(\xi)$ which maps $[0,1]$ to $\mathbb{R}_+$, 
the boundary-value problem (\ref{limit-Schr-op}) is rewritten in the form
\begin{equation}
\label{limit-Schr-op-infty}
\left. \begin{array}{l} 
-w''(\xi) = \nu\, {\rm sech}^6(\xi) w(\xi), \quad \xi \in \mathbb{R}_+, \\ 
w(0) = 0, \quad |w'(\xi)| \to 0 \;\; \mbox{\rm as} \;\; \xi \to \infty. \end{array} \right\} 
\end{equation}
The spatial domain of this boundary-value problem is truncated on $[0,10]$ resulting in the lowest eigenvalue at $\nu_0 \approx 7.29$. The asymptotic approximation on Fig. \ref{fig-bif} corresponds to $\gamma = 1 + \nu_0 \e^2$. 

\begin{figure}[htpb!]
	\centering
	\includegraphics[width=8cm,height=6cm]{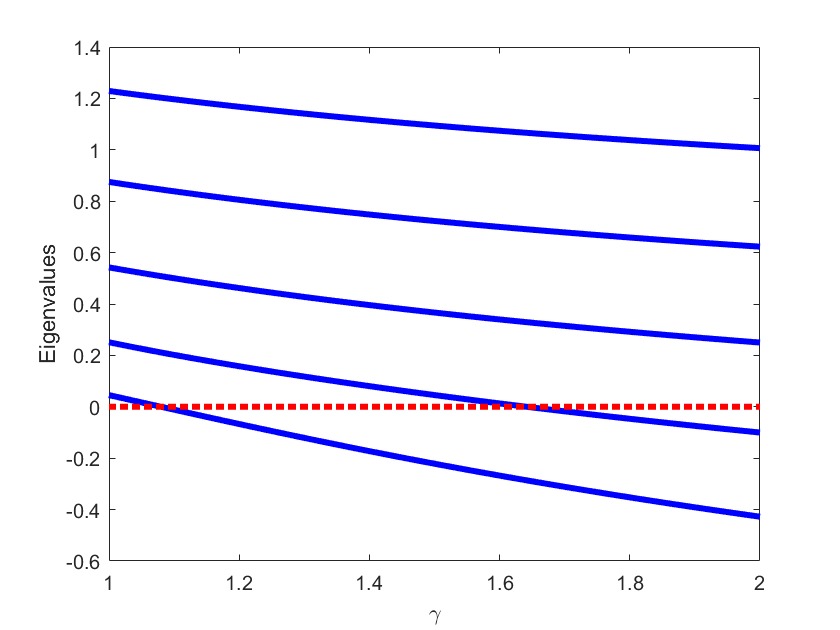}
	\caption{Eigenvalues of the Dirichlet problem for $L_{\gamma}$ on $\mathbb{R}_+$ versus $\gamma$.}
	\label{fig-eig}
\end{figure}

\subsection{Organization of the paper}

The remainder of this paper contains proofs of the main theorems. 

Section \ref{sec-2} gives the proof of the splitting formula $G_{\e}(\Psi) = F_{\e}(\eta_\e) + \e J_{\e}(\Phi)$ (Lemma \ref{sepen}) and 
the proof of Theorem \ref{gconv1} about $\Gamma$ convergence of 
$J_{\e}$ to $J_0$. 

Section \ref{classification} gives proofs 
of Theorems \ref{theorem-variational} and \ref{theorem-domain-walls} and  identify the parameter ranges for which the uncoupled, symmetric, and domain wall states are the global minimizers of the energy $G_\e$ in the energy space $\mathcal{E}$ and in the energy space with symmetry $\mathcal{E}_s$. 
The bifurcation at $\gamma = \gamma_0(\e)$ and the Schr\"{o}dinger operator 
$L_{\gamma}$ in (\ref{Schr-operator}) arise in the proof of Theorem \ref{theorem-domain-walls}. Existence of minimizers of $G_{\e}$ in $\mathcal{E}$ is given by Lemma \ref{DWalphaExist}. Lemma \ref{lemma-property} describes 
properties of the domain wall states (S3) as minimizers of $G_{\e}$ in $\mathcal{E}_s$. Lemma \ref{lem-uniqueness} gives uniqueness of the symmetric state (S2) as the positive solution of the coupled system (\ref{DWHPVR}) if $\gamma \in (0,1)$. Finally, we prove Theorem \ref{DWalpha} about the non-symmetric minimizers of $G_{\e}$ in $\mathcal{D}_{\alpha}$ with fixed $\alpha > 0$.

Section \ref{sec-bifurcations} investigates bifurcations from the one-parameter family of solutions (\ref{rotating-state}) at $\gamma=1$ establishing the proof of Theorem \ref{theorem-bifurcation}.  Lemma \ref{theorem-gamma} gives $\Gamma$-convergence of the energy functional $I_{\mu,\gamma}$ to $I_{\mu,1}$ as $\gamma \to 1^+$. The bifurcation $\mu = \mu_0$ and the boundary-value problem (\ref{limit-Schr-op}) arise in the proof of Theorem \ref{theorem-asymptotic} from analysis of $I_{\mu,1}$ at the symmetric state. Lemma \ref{lemma-gamma-uniqueness} gives uniqueness of the symmetric state 
if $\mu \in [\mu_0,\infty)$. Lemma \ref{lemma-normal-form} reports computations of the normal form for bifurcation of domain wall states and shows that the domain wall states bifurcate to $\mu < \mu_0$, where they are characterized as minimizers of $I_{\mu,1}$ if $\mu \in (0,\mu_0)$. 

As we have pointed out previously, the proof that the domain wall states (S3) only exist if $\gamma > \gamma_0(\e)$, where they are characterized as minimizers of $G_{\e}$, is open for further studies.

\vskip.1in{\bf Acknowledgements.}
The work of A.C. was partially supported by a grant from the Simons foundation \# 426318. V. S. acknowledges support by Leverhulme grant RPG-2018-438. 
D. P. is supported by the NSERC Discovery grant. 
\vskip.1in

\section{$\Gamma$-convergence of the energy functionals}
\label{sec-2}

The following splitting lemma was pioneered in \cite{LM} and later explored in \cite{IgMi1,IgMi2} in the context of vortices in $\RR^2$. We adopt it to the case of steady states that include the domain wall states in $\RR$.

\begin{lemma}\label{sepen} 
	Let $\Psi\in  \mathcal{E}$ and define $\Phi := \Psi/\eta_\e$, where $\eta_\e \in H^1(\RR;\mathbb{R}) \cap L^{2,1}(\RR;\mathbb{R})$ is the positive solution of \eqref{T1}. Then, $G_{\e}(\Psi) = F_{\e}(\eta_\e) + \e J_{\e}(\Phi)$, where $F_{\e}$ is given by (\ref{energy-F}) and 
\begin{equation}
\label{energy-J}
J_\e (\Phi) = \frac{1}{2\e} \int_\R \left[ \e^2 \eta_\e^2 (\phi_1')^2 + \e^2 \eta_\e^2 (\phi_2')^2 + \frac12 \eta_\e^4 (\phi_1^2 + \phi_2^2 -1)^2 + (\gamma-1) \eta_\e^4 \phi_1^2 \phi_2^2 \right]\, dx.
\end{equation}
\end{lemma}

\begin{proof}
First consider $\Psi = \eta_\e \Phi \in H^1 (\R; \R^2)$ with compact support. We conclude by integrating by parts and \eqref{T1} that
\begin{align*}
\int_\R \e^2 |\Psi'|^2 \, dx &= \e^2 \int_\R \left[ (\eta_\e')^2 |\Phi|^2 + \eta_\e^2 |\Phi'|^2 + \eta_\e \eta_\e' (|\Phi|^2-1)' \right] \,dx \\
&= \int_\R \left[ \e^2 \eta_\e^2 |\Phi'|^2 +\e^2 (\eta_\e')^2 + \eta_\e^2 (1- x^2 -\eta_\e^2) (|\Phi|^2 -1) \right] \,dx.
\end{align*}
For the  other parts of the energy $G_{\e}$, simple calculations yield
\begin{align*}
 \frac12\int_\R |\Psi|^4 \,dx &=  \frac12 \int_\R \eta_\e^4 |\Phi|^4 \,dx, \\
\int_\R (x^2-1) |\Psi|^2 \,dx  &= \int_\R (x^2-1) \eta_\e^2 |\Phi|^2 \,dx, \\
(\gamma-1) \int_\R \psi_1^2 \psi_2^2 \,dx  &= (\gamma-1) \int_\R  \eta_\e^4 \phi_1^2 \phi_2^2 \,dx.
\end{align*}
We notice that
$$
\eta_\e^2 (1- x^2 -\eta_\e^2) (|\Phi|^2 -1) + \frac12 \eta_\e^4 |\Phi|^4 +(x^2-1) \eta_\e^2 |\Phi|^2
= \frac12 \eta_\e^4 + \eta_\e^2 (x^2-1) + \frac12 \eta_\e^4 (|\Phi|^2-1)^2.
$$
Combining the above equalities we obtain the splitting $G_{\e}(\Psi) = F_{\e}(\eta_\e) + \e J_{\e}(\Phi)$ with $F_{\e}$ and $J_{\e}$ given by 
(\ref{energy-F}) and (\ref{energy-J}). 

For the general case, let $\chi:\R_+\to \R$ be a cutoff function such that $\chi\equiv 1$ in $[0,1]$,  $\chi\equiv 0$ in $[2,\infty)$, and $|\chi|,|\chi'|\leq 1$ almost everywhere. Let 
$$
\Psi_n(x):=\chi (|x|/n) \Psi (x), \quad n \in \mathbb{N}.
$$ 
Since $\Psi_n$ is compactly supported, we represent $\Psi_n = \eta_{\e} \Phi_n$ and obtain 
$$
G_{\e}(\Psi_n) = F_{\e}(\eta_\e) + \e J_{\e}(\Phi_n), \quad n\in \mathbb{N}.
$$ 
Next, we note that $|\Psi_n|\leq |\Psi|$, $|\Psi_n '|\leq |\Psi '|+|\Psi|,$ and that $\Psi_n \to \Psi$, $\Psi_n' \to \Psi'$ pointwise as $n\to \infty$. These facts, together with $\int_\R x^2 |\Psi|^2 dx <\infty$, allow us to apply the Dominated Convergence Theorem to conclude that 
$$
G_\e (\Psi_n)\to G_\e (\Psi) \quad \mbox{\rm as} \quad n \to \infty.
$$  
Now, since $\eta_\e \Phi_n\to \eta_\e \Phi$ and $\eta_\e\Phi_n '\to \eta_\e\Phi '$ pointwise as $n \to \infty$ and
$$
\eta_\e ^2 |\Phi_n '|^2\leq C (|\Psi '|^2 +|\Psi|^2+ |\eta'_\e|^2 |\Phi |^2),
\quad 
\eta_\e^4 (|\Phi_n|^2-1)^2\leq C (|\Psi|^2 +\eta_\e^2)^2.
$$ 
we obtain by the same Dominated Convergence Theorem that 
$$
J_\e (\Phi_n)\to J_\e (\Phi) \quad \mbox{\rm as} \quad n \to \infty,
$$ which concludes the proof of the splitting formula in the general case.
\end{proof}

\begin{remark}
Since positive $\eta_\e \in H^1(\RR;\mathbb{R}) \cap L^{2,1}(\RR;\mathbb{R})$ is a minimizer of $F_\e$, 
we have that $\Psi \in \mathcal{E}$ is a minimizer of $G_\e$ if and only if $\Phi \in H^1_{\rm loc}(\mathbb{R};\mathbb{R}^2)$ is a minimizer of $J_\e$. 
\end{remark}

\begin{remark}
	Expression (\ref{energy-J}) for $J_{\e}$ coincides with (\ref{GPU-reduced}) after the change of variables $x \mapsto z := x/\e$ and writing $\Phi$ as a function of $z$. In what follows we use the rescaled form of $J_\e$
given by (\ref{GPU-reduced}).
\end{remark}

We proceed with the proof of Theorem \ref{gconv1} about the $\Gamma$-convergence of $J_{\e}$ in (\ref{GPU-reduced}) to $J_0$ in (\ref{GPU-0}) as $\e \to 0$. \\

\begin{proof1}{\em of Theorem \ref{gconv1}.}
	We first note that as $\e \to 0$ we have $\eta_\e(\e \cdot) \to 1$ locally uniformly in $\R$. Indeed, using (\ref{P-3}) yields $\|\eta_\e'(\e \cdot)\|_\infty = \e \|\eta_\e'\|_\infty \leq C \e^{2/3}$ and since $\eta_\e(0) \to 1$ as $\e \to 0$, we have locally uniform convergence $\eta_\e(\e \cdot) \to 1$ on $\mathbb{R}$ as $\e \to 0$.\\

(i) Since $\gamma>1$ every term in $J_\e$ is non-negative and if $J_\e (\Phi_\e) \leq C < \infty$ then for every compact set $K \subset \R$ we clearly have $J_\e (\Phi_\e) \big|_K \leq C(K)$. Due to locally uniform convergence of $\eta_\e(\e \cdot)$ to $1$ we can find $\e_0$ small enough such that $\eta_\e(\e \cdot) \geq \frac12$ on $K$ for all $\e <\e_0$ and obtain
\begin{align}
 \int_K \left[ |\Phi_\e'|^2 +|\Phi_\e|^2  \right] \, dz \leq C(K).
\end{align}
Therefore, there is a subsequence (not relabelled) such that  $\Phi_\e \rightharpoonup \Phi$ in $H^1_{\rm loc}(\R;\R^2)$ as $\e \to 0$. \\
 
(ii) Let $\Phi_\e \in H^1_{\rm loc}(\R;\R^2)$ and $\Phi_\e \to \Phi$ strongly in $L^2_{\rm loc}(\R;\R^2)$ as $\e \to 0$. If $\liminf\limits_{\e \to 0} J_\e(\Phi_\e) = \infty$ there is nothing to prove as 
the inequality (\ref{inequality-1}) is trivially satisfied and therefore we can assume $\liminf\limits_{\e \to 0} J_\e(\Phi_\e) <\infty$. Taking a subsequence $\Phi_\e$ (not relabelled) such that 
$$
\lim_{\e \to 0}  J_\e(\Phi_\e) = \liminf_{\e \to 0} J_\e(\Phi_\e) <\infty
$$
we know that there is $\e_0>0$ such that for all $0<\e <\e_0$ we have $J_\e(\Phi_\e) \leq C < \infty$ with $C>0$ independent of $\e$. Therefore, using part (i), there is a subsequence (not relabelled) such that $\Phi_\e \rightharpoonup \Phi$ in $H^1_{\rm loc}(\R;\R^2)$ as $\e \to 0$. Fixing $R>0$ and using non-negativity of terms in $J_\e$ together with the lower semicontinuity of the $H^1$-norm, we have
\begin{align*}
\liminf_{\e \to 0} J_\e(\Phi_\e) &\geq \liminf_{\e \to 0} J_\e(\Phi_\e) \Big|_{[-R,R]} \\
 &\geq \frac12 \int_{[-R,R]} \left( |\phi_{1}'|^2 + |\phi_{2}'|^2 + \frac12 (\phi_{1}^2 + \phi_{2}^2 -1)^2 + (\gamma-1) \phi_{1}^2 \phi_{2}^2 \right) \, dz.
\end{align*}
Taking the limit as $R \to \infty$ and using the Monotone Convergence Theorem we obtain $\liminf\limits_{\e \to 0} J_\e(\Phi_\e) \geq J_0 (\Phi)$. \\

(iii) Let $\Phi \in H^1_{\rm loc}(\R;\R^2)$. If $J_0 (\Phi)=\infty$, then the inequality (\ref{inequality-2}) is trivially satisfied by taking $\Phi_\e \equiv \Phi$ and therefore we assume that $J_0(\Phi)<\infty$. In this case we also define $\Phi_\e \equiv \Phi$ and clearly have $\Phi_\e \to \Phi$ in  $H^1_{\rm loc}(\R;\R^2)$ as $\e \to 0$. Since $0< \eta_\e(\e \cdot) \leq 1$ follows from (\ref{P-1}), we have $J_\e(\Phi) \leq J_0(\Phi)$ and hence $\limsup\limits_{\e \to 0} J_\e (\Phi_\e) =  \limsup\limits_{\e \to 0} J_\e(\Phi) \leq  J_0 (\Phi)$.
\end{proof1}

\section{Classification of minimizers of $G_\e$}\label{classification}

Here we characterize global minimizers of $G_\e$ in the energy space $\mathcal{E}$ for small $\e>0$. We start by showing that minimizers of $G_{\e}$ always exist.

\begin{lemma}\label{DWalphaExist}
	For any $\e>0$ and $\gamma>0$, there exists a minimizer of the energy $G_\e$ in $\mathcal E$.
\end{lemma}

\begin{proof}
It is clear from (\ref{GPU}) that there is a fixed constant $C>0$ (independent of $\e$) such that for all $\Psi \in \mathcal E$ we have $G_\e(\Psi) \geq -C$. Therefore, we can take a minimizing sequence $\{ \Psi_n \}_{n \in \mathbb{N}}$ such that 
$$
	G_\e(\Psi_n) \to \inf_{\Psi \in \mathcal E} G_\e \quad 
	\mbox{\rm as} \quad n \to \infty.
$$
Since $\inf_{\Psi \in \mathcal E} G_\e < 0$, we can choose $\{ \Psi_n \}_{n \in \mathbb{N}}$ such that $G_\e(\Psi_n) \leq 0$ and therefore $\{ \Psi_n \}_{n \in \mathbb{N}}$ is uniformly bounded in $\mathcal E$. This implies (possibly on a subsequence) that as $n \to \infty$, 
	$$
	\Psi_n \rightharpoonup \Psi  \hbox{ in } \mathcal E
	$$
	and
	$$
	\Psi_n \to \Psi \hbox{ in } C([-R,R];\R^2) \hbox{ for any fixed } R>0.
	$$
It follows from (\ref{GPU}) that since the energy density of $G_\e$ is non-negative outside $[-1,1]$, we have $G_\e(\Psi_n) \geq G_\e(\Psi_n) \big|_{[-R;R]}$ for $R>1$. Hence, we have 
	$$
	\liminf_{n \to \infty} G_\e(\Psi_n) \geq \liminf_{n \to \infty} G_\e(\Psi_n) \big|_{[-R,R]} = G_\e(\Psi) \big|_{[-R,R]}.
	$$
	Passing to the limit as $R \to \infty$ and using the Monotone Convergence Theorem, we obtain $\liminf G_\e(\Psi_n) \geq G_\e(\Psi)$ with $\Psi \in \mathcal E$. Therefore $\Psi$ is a minimizer.
\end{proof}

\begin{remark}
	\label{remark-symmetry}
Imposing the symmetry constraint in $\mathcal{E}_s$ does not change 
the argument, hence Lemma \ref{DWalphaExist} holds in $\mathcal{E}_s$ as well. 
Indeed, due to the symmetry, the minimization problem can be reduced to the half line $\R_+$ subject to the condition $\psi_1(0)=\psi_2(0)$. This condition can be applied due to Sobolev embedding of $H^1(\RR_+)$ into $C^0(\RR_+) \cap L^{\infty}(\RR_+)$.
\end{remark}

Regarding the minimizers of the energy $G_{\e}$ in $\mathcal{E}_s$, the following lemma gives the precise property of minimizers. 

\begin{lemma}
	\label{lemma-property}
Let $\Psi=(\psi_1, \psi_2)$ be a minimizer of $G_\e$ in $\mathcal{E}_s$ 
with $\psi_1(0) = \psi_2(0) > 0$. Then, either 
$\psi_1(x) > \psi_2(x) > 0$ or $\psi_2(x) > \psi_1(x) > 0$ or $\psi_1(x) =\psi_2(x)$ for $x > 0$.
\end{lemma}

\begin{proof}
If $\Psi =(\psi_1, \psi_2)$ is a minimizer of the energy $G_\e$ in $\mathcal{E}_s$, then it is clear that $\Psi' = (|\psi_1|, |\psi_2|)$ is also a minimizer of the energy $G_\e$ in $\mathcal{E}_s$. Both satisfy the Euler-Lagrange equation \eqref{DWHPVR} and hence they are smooth. Using the Strong Maximum Principle \cite[Theorem 2.10]{hanlin} for $|\psi_1|$ and $|\psi_2|$ we deduce that they are strictly positive and since it is assumed that $\psi_1(0) = \psi_2(0) >0$, then for any minimizer $\Psi = (\psi_1,\psi_2)$ we have $\psi_1(x)>0$ and $\psi_2(x)>0$ for $x > 0$.
	
Next, we show that $\psi_1$ and $\psi_2$ either coincide or have no intersections in $(0,\infty)$. Assume there is an interval $(a,b)$ such that $\psi_1(a)=\psi_2(a)$ and $\psi_1(b)=\psi_2(b)$. Due to symmetry of the energy with respect to replacing $\psi_1$ by $\psi_2$ on the interval $(a,b)$ we can define a new minimizer $\Psi$ by exchanging $\psi_1$ and $\psi_2$ on $(a,b)$. This minimizer satisfies the Euler-Lagrange equation \eqref{DWHPVR} and hence 
it is smooth. Therefore, we can always construct a minimizer such that the difference $\xi := \psi_1 -\psi_2$ does not change its sign on $(0,\infty)$. It follows from (\ref{DWHPVR}) that $\xi$ is smooth, satisfies  $\xi(0) = \psi_1(0) - \psi_2(0) = 0$, and
	\begin{align}
	\label{xi-eq}
	-\e^2 \xi'' + (x^2 + \psi_1^2 + \psi_2^2 -1) \xi - (\gamma-1) \psi_1 \psi_2 \xi =0.
	\end{align}
Using the Strong Maximum Principle again, we deduce that either $\xi(x)>0$ or $\xi(x)<0$ or $\xi(x) =0$ for all $x \in (0,\infty)$.
\end{proof}

\begin{remark}
	According to the definitions of the steady states, minimizers with $\psi_1 \equiv \psi_2$ correspond to the symmetric state (S2) and the minimizers with either 
	$\psi_1(x) > \psi_2(x) > 0$ or $\psi_2(x) > \psi_1(x) > 0$ for $x > 0$ 
	correspond to the domain wall states (S3).
\end{remark}	

Next we prove uniqueness of positive solutions of (\ref{DWHPVR}) 
for $\gamma \in (0,1)$ given by the symmetric state (S2).

\begin{lemma}
	\label{lem-uniqueness}
The symmetric state $\Psi = (1+\gamma)^{-1/2} (\eta_{\e},\eta_{\e})$ is the only positive solution of system \eqref{DWHPVR} for $\gamma \in (0,1)$.	
\end{lemma}

\begin{proof}
Define the functional $E(W)$ for $W=(w_1, w_2)$ with $w_1>0$ and $w_2>0$ by 
$$
E(W) :=G_\varepsilon (\sqrt{w_1},\sqrt{w_2}).
$$ 
Let $V=(v_1,v_2)$ be such that there is $t_0>0$ such that $w_1+t v_1>0$ and $w_2+t v_2>0$ for $0<t<t_0$. A direct calculation shows that for any admissible $W$ and $V$, we have
\begin{eqnarray*}
	\frac{d^2}{dt^2} E(W+tV) &=& \frac{\e^2}{2} \int_{\mathbb{R}} \left[ \frac{(v_1' w_1 -v_1 w_1')^2}{2(w_1+tv_1)^3}+\frac{(v_2' w_2 -v_2 w_2')^2}{2(w_2+tv_2)^3} \right] dx \\
	&& +\int_{\mathbb{R}} \left[ v_1^2+v_2^2+2\gamma v_1v_2 \right] dx,
\end{eqnarray*}
which is strictly positive definite if $|\gamma| < 1$. In particular, for any positive $W_1 \neq W_2,$ the function $f_{W_1,W_2}(t):=E(W_1+t (W_2-W_1))$ satisfies $(f_{W_1,W_2})''(t)>0,$ for $t\in (0,1).$

Now, let $\Psi = (\sqrt{w}_1,\sqrt{w}_2)$ be a positive solution of system \eqref{DWHPVR} and let $\Psi' = (\sqrt{v_1},\sqrt{v_2})$ be another positive function such that $\Psi \neq \Psi'$. Since $\Psi$ is a critical point of $G_{\e}$, we have that $(f_{W,V})' (0)=0.$ Since, $(f_{W,V})''(t)>0,$ for $t\in (0,1),$ this implies that
$$ 
G_\varepsilon (\Psi')-G_\varepsilon (\Psi) = \int_{0}^{1} (f_{W,V})' (t) \, dt>0.
$$
Thus, $\Psi'$ is not a critical point of $G_\e$, otherwise we could reverse the roles of $\Psi$ and $\Psi'$ to conclude $G_\varepsilon (\Psi)-G_\varepsilon (\Psi')>0$, which is a contradiction.
\end{proof}

We can now proceed with the proofs of Theorems \ref{theorem-variational} and \ref{theorem-domain-walls}.\\

\begin{proof1}{\em of Theorem \ref{theorem-variational}.}
	By virtue of Lemma \ref{sepen}, the problem of minimizing $G_\e$ in $\mathcal{E}$ is equivalent to the problem of minimizing $J_\e$ in $H^1_{\rm loc}(\RR;\RR^2)$.
	Moreover, it is clear that
	\[
	J_\e(\Phi) \geq  \frac{1}{2} \int_{\RR} \eta_\e(\e \cdot)^4 \left[ \frac{1}{2} (\phi_1^2+\phi_2^2-1)^2 + (\gamma-1) \phi_1^2 \phi_2^2\right] dz.
	\]
	We proceed separately for $\gamma \in (0,1)$ and $\gamma \in (1,\infty)$. \\
	
	\underline{{\it Case $0<\gamma<1.$}} In this case, the quadratic form
	\be\label{Wgamma}
	W_\gamma(\phi_1,\phi_2):=\frac{1}{2}(\phi_1^2 + \phi_2^2 -1)^2 + (\gamma-1) \phi_1^2 \phi_2^2
	\ee
	achieves a minimum at $\Phi = (1+\gamma)^{-1/2} (1,1)$. To see this, it suffices to write
	$$
	\phi_1 = r \cos \theta, \quad \phi_2 = r \sin \theta
	$$
	and obtain the lower bound:
	\[
	W_\gamma (\phi_1,\phi_2)\geq \frac{1}{2} (r^2-1)^2+ \frac{1}{4} (\gamma-1)r^4 \geq 	-\frac{(1-\gamma)}{2(1+\gamma)}, 
	\]
	Therefore,
	\[J_\e(\Phi) \geq  -\frac{(1-\gamma)}{4(1+\gamma)} \int_{\RR} \eta_\e(\e z)^4 dz.
	\]
	Since the minimum of $J_{\e}(\Phi)$ is achieved at $\Phi = (1+\gamma)^{-1/2} (1,1)$, the symmetric state (S2) given by  (\ref{positive-state}) is a global minimizer of $G_{\e}$ due to the product form (\ref{transformation}) and Lemma \ref{sepen}. It is the only positive solution of system (\ref{DWHPVR}) by Lemma \ref{lem-uniqueness}.
	
	If $\Phi = (1,0)$ or $\Phi = (0,1)$, then $J_{\e}(\Phi) = 0$. Hence 
	the uncoupled states (S1) given by (\ref{uncoupled-state}) are not minimizers of $G_{\e}$.
	We show that these states are saddle points of $G_{\e}$ by computing the Hessian operator for $G_{\e}(\Psi)$ at $\Psi = (\eta_{\e},0)$:
	\begin{equation}
	\label{operator-H-uncoupled}
	G_{\e}''(\eta_{\e},0) = \begin{pmatrix} -\e^2\p_x^2 + x^2 - 1 + 3 \eta_{\e}^2 & 0\\
	0 & -\e^2\p_x^2 + x^2 - 1 + \gamma \eta_{\e}^2 \end{pmatrix} =: 
	\begin{pmatrix} H^{(1)}_{\e} & 0\\
	0 & H^{(2)}_{\e} \end{pmatrix}.
	\end{equation}
The diagonal entry $H^{(1)}_{\e}$ coincides with the Schr\"{o}dinger operator
	$$
	L_+ := -\e^2\p_x^2 + x^2 - 1 + 3 \eta_{\e}^2
	$$
considered in the context of the ground state of the scalar Gross--Pitaevskii 
equation \cite{GaPe09,GaPe}. By Theorem 2 in \cite{GaPe} (see also Lemma 2.3 in \cite{GaPe09}), the spectrum of $L_+$
in $L^2(\RR)$ is strictly positive and bounded away from zero by $C \e^{2/3}$ as $\e \to 0$ with an $\e$-independent constant $C > 0$.

The diagonal entry $H^{(2)}_{\e}$ depends on $\gamma$ and should be compared with the Schr\"{o}dinger	operator
	$$
	L_- :=  -\e^2\p_x^2 + x^2 - 1 + \eta_{\e}^2
	$$
which was also considered in \cite{GaPe09,GaPe} within the scalar theory. 
The spectrum of $L_-$ in $L^2(\RR)$ is non-negative with the lowest eigenvalue 
at $0$ due to the exact solution $L_- \eta_{\e} = 0$ which is
	equivalent to the stationary equation (\ref{T1}). 
	The rest of the spectrum is strictly positive and bounded away from zero by $C \e^2$ as $\e \to 0$ with an $\e$-independent constant $C > 0$ (see Lemma 2.1 in \cite{GaPe09}).
By Sturm's Comparison Theorem \cite[Section 5.5]{Teschl}, 
the spectrum of  $H^{(2)}_{\e}$ in $L^2(\RR)$ has at least one negative eigenvalue for every $\gamma \in (0,1)$. Therefore, the uncoupled state
(S1) are saddle points of $G_{\e}$ if $\gamma \in (0,1)$.\\
	
	\underline{{\it Case $\gamma>1.$}} Since $J_{\e}(\Phi) \geq 0$ for $\gamma > 1$ and $J_{\e}(\Phi) = 0$ if $\Phi = (1,0)$ or $\Phi = (0,1)$,
	the uncoupled states (S1) are minimizers of $G_{\e}$. 
	This is in agreement with the Hessian operator (\ref{operator-H-uncoupled}), 
for which both diagonal entries are strictly positive for $\gamma > 1$. Moreover, the uncoupled states (S1) are the only non-negative minimizers of $G_{\e}$ because all terms in $J_{\e}(\Phi)$ are positive for $\gamma > 1$ so that the minimizers must satisfy $(\phi_1')^2 + (\phi_2')^2 = 0$, $\phi_1^2 + \phi_2^2 = 1$, and $\phi_1^2 \phi_2^2 = 0$ almost everywhere on $\mathbb{R}$ which yield either  $\Phi = (1,0)$ or $\Phi = (0,1)$ for $\Phi \in H^1_{\rm loc}(\RR;\RR^2)$.

Next we show that the symmetric state (S2) is a saddle
point of $G_{\e}$ by computing the Hessian operator for $G_{\e}(\Psi)$ at $\Psi = (1+\gamma)^{-1/2} (\eta_{\e},\eta_{\e})$:
	\begin{eqnarray}
	\nonumber
	G_{\e}''(\Psi) & = & \begin{pmatrix}-\e^2\p_x^2 + x^2 + 3 \psi_1^2 + \gamma \psi_2^2 - 1 & 2\gamma \psi_1 \psi_2\\
	2\gamma \psi_1 \psi_2 & -\e^2\p_x^2 + x^2 + \gamma \psi_1^2 + 3 \psi_2^2 - 1 \end{pmatrix} \\
	\label{operator-H}
	& = & \begin{pmatrix}-\e^2\p_x^2 + x^2 + \frac{3+\gamma}{1+\gamma} \eta_{\e}^2 - 1 & \frac{2\gamma}{1+\gamma} \eta_{\e}^2 \\
	\frac{2 \gamma}{1+\gamma} \eta_{\e}^2  & -\e^2\p_x^2 + x^2 + \frac{3+\gamma}{1+\gamma} \eta_{\e}^2  - 1 \end{pmatrix}.
	\end{eqnarray}
With an elementary transformation, the Hessian operator $G_{\e}''(\Psi)$ 
is diagonalized as follows
\begin{eqnarray}
\nonumber
&& 
\frac{1}{2} \begin{pmatrix} 1 & 1 \\ 1 & -1 \end{pmatrix} 	G_{\e}''(\Psi) 
\begin{pmatrix} 1 & 1 \\ 1 & -1 \end{pmatrix} \\
\label{rotation}
&& = 
\begin{pmatrix} -\e^2\p_x^2 + x^2 - 1 + 3 \eta_{\e}^2 & 0 \\ 0 & -\e^2\p_x^2 + x^2  - 1 + \eta_{\e}^2 + 2\frac{1-\gamma}{1+\gamma} \eta_{\e}^2 \end{pmatrix}.
\end{eqnarray}
The first diagonal entry coincides with $L_+$, which is strictly positive in $L^2(\mathbb{R})$. The second diagonal entry depends on $\gamma$ and can be compared with $L_- .$ Since the spectrum of $L_-$ in $L^2(\mathbb{R})$ has a simple zero eigenvalue, by the same Sturm's Comparison Theorem, there exists at least one negative eigenvalue of $G''_{\e}(\Psi)$ if $\gamma > 1$. Therefore, the symmetric state
(S2) is a saddle point of $G_{\e}$ if $\gamma > 1$.
\end{proof1}

\begin{remark}
	\label{remark-diagonalization}
	With the use of the decomposition $\mathfrak{v} := \eta_{\e} v$ as in Lemma \ref{sepen}, the quadratic forms for the diagonal operators in (\ref{rotation}) can be rewritten as 
\begin{eqnarray}
&&	\int_{\mathbb{R}} \left[ \e^2 (\mathfrak{v}')^2	+ (x^2 - 1 + 3 \eta_e^2) \mathfrak{v}^2 \right] dx = \int_{\mathbb{R}} \left[ \e^2 \eta_{\e}^2 (v')^2 
+ 2 \eta_{\e}^4 v^2 \right] dx, 
\label{form-1} \\
\nonumber 
&&	\int_{\mathbb{R}} \left[ \e^2 (\mathfrak{v}')^2	+ (x^2 - 1 + \eta_e^2) \mathfrak{v}^2 + 2 \frac{1-\gamma}{1+\gamma} \mathfrak{v}^2 \right] dx \\
&& \qquad \qquad \qquad  = \int_{\mathbb{R}} \left[ \e^2 \eta_{\e}^2 (v')^2 
+ 2 \frac{1-\gamma}{1+\gamma}\eta_{\e}^4 v^2 \right] dx. 
\label{form-2}
\end{eqnarray}
It follows from positivity of the quadratic form (\ref{form-1}) that $L_+$ is strictly positive, 
in agreement with \cite{GaPe09,GaPe}. On the other hand, the quadratic form (\ref{form-2}) is not positive if $\gamma > 1$ as follows from the trial function $v = 1$.
\end{remark}

\vspace{0.25cm}

\begin{proof1}{\em of Theorem \ref{theorem-domain-walls}.}
Let us now consider minimizers of $G_{\e}$ in the energy space $\mathcal{E}_s$ with the symmetry constraint: 
\begin{equation}
\label{symmetry}
\psi_1(x) = \psi_2(-x), \qquad x \in \mathbb{R}.
\end{equation} 
By Remark \ref{remark-symmetry}, there is a global minimizer of $G_{\e}$ in $\mathcal{E}_s$ for every $\gamma > 0$. By Theorem \ref{theorem-variational}, the symmetric state (S2) is the minimizer of $G_{\e}$ in $\mathcal{E}_s$ for $\gamma \in (0,1)$. 

We will show that there exists $\gamma_0(\e) \in (1,\infty)$ such that 
the symmetric state (S2) is a minimizer of $G_{\e}$ in $\mathcal{E}_s$ 
for $\gamma \in (0,\gamma_0(\e)]$ and a saddle point for $\gamma \in (\gamma_0(\e),\infty)$. To do so, we recall the Hessian operator (\ref{operator-H}) at the symmetric state (S2). If $\tilde{\Psi} = (\tilde{\psi}_1,\tilde{\psi}_2)$ is a perturbation of 
$\Psi$ in $\mathcal{E}_s$, then it satisfies the symmetry (\ref{symmetry}) so that the perturbation in $\mathcal{E}_s$ can be considered on the half-line $\mathbb{R}_+$ subject to the boundary condition 
$$
\tilde{\psi}_1(0) - \tilde{\psi}_2(0) = 0.
$$ 
With the symmetry in $\mathcal{E}_s$, the second variation of $G_\e$ at $\Psi$ in the direction $\tilde{\Psi}$ is given by 
\begin{align*}
\int_0^\infty \Big[ \e^2 (\tilde{\psi}_1')^2 + \e^2 (\tilde{\psi}_2')^2 + (x^2-1) (\tilde{\psi}_1^2 + \tilde{\psi}_2^2) + \frac{3+\gamma}{1+\gamma} \eta_{\e}^2 (\tilde{\psi}_1^2 + \tilde{\psi}_2^2) 
+ \frac{4 \gamma}{1+\gamma} \eta_{\e}^2 \tilde{\psi}_1 \tilde{\psi}_2
\Big].
\end{align*}
which is rewritten by using the decomposition $\tilde{\Psi} = \eta_{\e} \tilde{\Phi}$ (as in Remark \ref{remark-diagonalization}) in the form
\begin{align*}
&\int_0^{\infty} \Big[ \e^2 \eta_{\e}^2 (\tilde{\phi}_1')^2 + \e^2 \eta_{\e}^2 (\tilde{\phi}_2')^2
 + \frac{2}{1+\gamma} \eta_{\e}^4 (\tilde{\phi}_1^2 + \tilde{\phi}_2^2 + 2 \gamma \tilde{\phi}_1 \tilde{\phi}_2)  
\Big] \\
&= \int_0^{\infty} \Big[ \e^2 \eta_{\e}^2 (\tilde{\phi}_1')^2 + \e^2 \eta_{\e}^2 (\tilde{\phi}_2')^2
 + 2\frac{1-\gamma}{1+\gamma} \eta_{\e}^4 (\tilde{\phi}_1^2 + \tilde{\phi}_2^2) + 
  \frac{2 \gamma}{1+\gamma}\eta_\e^4 ( \tilde{\phi}_1 + \tilde{\phi}_2)^2 
\Big] ,
\end{align*}
subject to the boundary condition $\tilde{\phi}_1(0) - \tilde{\phi}_2(0) = 0$. 
It is clear from (\ref{rotation}) that the negative subspace of the quadratic form is given by the reduction $v := \tilde{\phi}_1 = -\tilde{\phi}_2$, which satisfies the Dirichlet condition $v(0) = 0$. 
This recovers the quadratic form in (\ref{form-2}), which is defined by the Schr\"{o}dinger operator
\begin{equation}
L_{\gamma} := L_- + 2 \frac{1-\gamma}{1+\gamma} \eta_{\e}^2
\end{equation}
on $H^1_0(\RR_+) \cap L^{1,2}(\RR_+) \subset L^2(\RR_+)$, 
where $H^1_0(\RR_+)$ encodes the Dirichlet condition at $x = 0$. Since the spectrum of the Dirichlet problem for $L_-$ in $L^2(\RR_+)$ consists of strictly positive eigenvalues and the potential of $L_{\gamma}$ is decreasing function of $\gamma$,  eigenvalues of $L_{\gamma}$ are continuous and strictly decreasing. Since the positive eigenvalues of $L_-$ on $H^1_0(\RR_+) \cap L^{1,2}(\RR_+) \subset L^2(\RR_+)$ are $\mathcal{O}(\e^2)$ close to $0$, there exists $\e_0 > 0$ such that for every $\e \in (0,\e_0)$ there exists $\gamma_0(\e) \in (1,\infty)$ such that $L_{\gamma}$ is strictly positive for $\gamma \in (1,\gamma_0(\e))$ and admits at least one negative eigenvalue for $\gamma \in (\gamma_0(\e),\infty)$. 

Next, we show that the symmetric state (S2) is a global minimizer of the energy $G_{\e}$ for $\gamma \in [1,\gamma_0(\e)]$. To show this, it suffices to show that $\Phi = (1+\gamma)^{-1/2} (1,1)$ is the global minimizer of the functional $J_\e$ defined in \eqref{energy-J}. For any perturbation $\tilde{\Phi} \in \mathcal E_s$ we obtain 
	\begin{align*}
J_\e (\Phi + \tilde{\Phi}) -J_\e(\Phi) &= \eps^{-1} \int_0^\infty \left( 
\e^2 \eta_\e^2 |\tilde{\Phi}'|^2 + \frac12 \eta_\e^4\left[ (|\Phi + \tilde{\Phi}|^2 -1)^2 - (|\Phi|^2 -1)^2 \right] \right.\\
	& \qquad  \left. + (\gamma-1) \eta_\e^4 \left[ (\phi_1 + \tilde{\phi}_1)^2 (\phi_2+ \tilde{\phi}_2)^2 - \phi_1^2 \phi_2^2 \right]\, \right) dx.
	\end{align*}
A straightforward computation yields
	\begin{align*}
J_\e (\Phi + \tilde{\Phi}) -J_\e(\Phi) &= \eps^{-1} \int_0^\infty \left( \e^2 \eta_\e^2 |\tilde{\Phi}'|^2 - 2\frac{\gamma-1}{\gamma+1} \eta_\e^4 |\tilde{\Phi}|^2 \right. 
	+\frac12 \eta_\e^4(|\tilde{\Phi}|^2 + 2 \Phi \cdot \tilde{\Phi})^2 \\
& \qquad +\left.  (\gamma-1) \eta_\e^4  \left[ (\tilde{\phi}_1 \tilde{\phi}_2 +\Phi \cdot \tilde{\Phi})^2  + (\Phi \cdot \tilde{\Phi})^2)\right] \right) dx. 
	\end{align*}  
	Since the second variation of $J_{\e}$ at $\Phi$ is non-negative in $\mathcal{E}_s$ for all $\gamma \leq \gamma_0(\e)$, we have 
	$$
	\int_0^\infty \left( \e^2 \eta_\e^2 |\tilde{\Phi}'|^2 - 2 \frac{\gamma-1}{\gamma+1} \eta_\e^4 |\tilde{\Phi}|^2 \right) dx \geq 0.
	$$
	Since all other terms are positive for $\gamma \geq 1$, we have 
$J_\e (\Phi + \tilde{\Phi}) -J_\e(\Phi) \geq 0$ for $\gamma \in [1,\gamma_0(\e)]$. Therefore, the symmetric state (S2) is a global minimizer of $G_{\e}$ in $\mathcal{E}_s$ for  $\gamma \in [1,\gamma_0(\e)]$. 
Augmented with the result of Theorem 2, it is a global minimizer of $G_{\e}$ 
in $\mathcal{E}_s$ for $\gamma \in (0,\gamma_0(\e)]$.
Since it is a saddle point for $\gamma \in (\gamma_0(\e),\infty)$, the pair of domain wall states (S3) are global minimizers of $G_{\e}$ in $\mathcal{E}_s$ in this case. 

It remains to show that $\gamma_0(\e) \to 1$ as $\e \to 0$. Consider the trial 
function $v$ satisfying the Dirichlet condition:
$$
v(x) = \left\{ \begin{array}{ll} 
4 x^2, \quad & 0\leq x \leq \frac12, \\
1, \quad &\frac12 < x <\infty.
\end{array}
\right.
$$
Plugging $v$ into (\ref{form-2}) on $\R_+$, 
we obtain
\begin{align*}
& \int_{0}^{\infty} \left[ \e^2 \eta_{\e}^2 (v')^2 
+ 2 \frac{1-\gamma}{1+\gamma}\eta_{\e}^4 v^2 \right] dx \\
& = 32 \int_0^\frac12 x^2 \eta_\e^2 \Big[ 2 \e^2 - \frac{\gamma-1}{\gamma + 1} x^2 \eta_\e^2  \Big] dx - 2 \frac{\gamma-1}{\gamma + 1} \int_\frac12^{\infty} \eta_{\e}^4 dx.
\end{align*}
Since $\eta_{\e} \to \eta_0$ as $\e \to 0$ and $\eta_0$ is given by (\ref{TF-limit}), it is clear that the quadratic form above is strictly negative for every $\gamma > 1$ as $\e \to 0$. Consequently, $\gamma_0(\e) \to 1$ as $\e \to 0$. 
\end{proof1}

\begin{remark}
Domain wall states (S3) bifurcate from the symmetric state (S2) by means of the local (pitchfork) bifurcation at $\gamma = \gamma_0(\e)$ via zero eigenvalue of the Hessian operator. Computing the normal form of the local bifurcation does not give a sign-definite coefficient which would imply that the domain wall states bifurcate for $\gamma > \gamma_0(\e)$. Therefore, we are not able to claim that the domain wall states do not exist in the region where $\gamma \in (1,\gamma_0(\e))$.
\end{remark}

\begin{remark}
	As $\e \to 0$, the bifurcation threshold $\gamma_0(\e)$ converges to $1$ 
	and it is known from \cite{ABCP} that the domain wall states of the homogeneous system (\ref{hc}) exist for every $\gamma > 1$. The minimizers of $G_\e$ are related to the minimizers of $J_\e$ through the transformation formula \eqref{transformation} and therefore, the domain wall states (S3) are related to the domain wall solutions of the coupled system \eqref{hc} with the boundary conditions \eqref{plus} via $\Gamma$-convergence of $J_{\e}$ to $J_0$ by Theorem \ref{gconv1}.
\end{remark}

Finally, we prove that if the non-symmetric critical points of $G_{\e}$ in $\mathcal{D}_{\alpha}$ exist, they are global minimizers of $G_{\e}$ in $\mathcal{D}_{\alpha}$. We employ the decomposition trick allowing us to significantly simplify the analysis of the second variation and the energy excess (see e.g. \cite{INSZ15, INSZ18}).

\vspace{0.25cm}
	
\begin{proof1}{\em of Theorem \ref{DWalpha}.}
Assume that there exists a critical point $\Psi = (\psi_1,\psi_2) \in \mathcal{D}_{\alpha}$ of the energy $G_{\e}$ satisfying $\psi_1(x) > \psi_2(x) > 0$ for all $x > 0$. For any perturbation $\tilde{\Psi} = (\tilde{\psi}_1,\tilde{\psi}_2) \in \mathcal{E}_{s}$ satisfying $\tilde{\psi}_1(0) = \tilde{\psi}_2(0) = 0$, we have 
		\begin{align*}
		G_\e (\Psi+\tilde{\Psi}) - G_\e (\Psi) &= \int_0^\infty \Big[ \e^2 (\tilde{\psi}_1')^2 + \e^2 (\tilde{\psi}_2')^2+ 
		(x^2+\psi_1^2+\psi_2^2-1) (\tilde{\psi}_1^2 + \tilde{\psi}_2^2) \\ 
		&+ \frac12 (2 \psi_1 \tilde{\psi}_1 + 2 \psi_2 \tilde{\psi}_2 + \tilde{\psi}_1^2 + \tilde{\psi}_2^2)^2 \\
		&+ (\gamma-1)[(2 \psi_1 \tilde{\psi}_1 + \tilde{\psi}_1^2)(2 \psi_2 \tilde{\psi}_2 + \tilde{\psi}_2^2) + \psi_1^2  \tilde{\psi}_2^2 + \psi_2^2 \tilde{\psi}_1^2]
		\Big] dx
		\end{align*}
Since a critical point $\Psi $ is smooth in the Euler--Lagrange equation (\ref{DWHPVR}) and since $\xi:=\psi_1-\psi_2 > 0$ on $(0,\infty)$ holds by the assumption of the theorem, we can write $\tilde{\Psi} = \xi \tilde{\Phi}$ with $\tilde{\Phi} \in H^1(\R_+;\R^2)$. Since $\xi(0) = \psi_1(0) - \psi_2(0) = 0$, 
the Dirichlet condition $\tilde{\psi}_1(0) = \tilde{\psi}_2(0) = 0$ is satisfied with arbitrary boundary values of $\tilde{\phi}_1(0)$ and $\tilde{\phi}_2(0)$. By using equation (\ref{xi-eq}) and integration by parts, we obtain 
		\begin{align*}
		\int_0^\infty \e^2 |\tilde{\Psi}'|^2 dx &=  \e^2 \int_0^\infty \left[ (\xi')^2 |\tilde{\Phi}|^2 + \xi^2 |\tilde{\Phi}'|^2 + \xi' \xi (|\tilde{\Phi}|^2)' \right] dx \\
		&= \e^2 \int_0^\infty  \left[ \xi^2 |\tilde{\Phi}'|^2 - \xi'' \xi |\tilde{\Phi}|^2 \right] dx.
		\end{align*}
		Subsituting (\ref{xi-eq}) gives now
		\begin{align*}
		G_\e (\Psi+\tilde{\Psi}) - G_\e (\Psi) = \int_0^\infty \Big[ \e^2 \xi^2 (\tilde{\phi}_1')^2 + \e^2 \xi^2 (\tilde{\phi}_2')^2 +  \frac12 (2 \psi_1 \tilde{\psi}_1 + 2 \psi_2 \tilde{\psi}_2 + \tilde{\psi}_1^2 + \tilde{\psi}_2^2)^2&&  \\
		+ (\gamma-1)[(2 \psi_1 \tilde{\psi}_1 + \tilde{\psi}_1^2)(2 \psi_2 \tilde{\psi}_2 + \tilde{\psi}_2^2) + \psi_1^2  \tilde{\psi}_2^2 + \psi_2^2 \tilde{\psi}_1^2 + \psi_1\psi_2 (\tilde{\psi}_1^2 + \tilde{\psi}_2^2)]
		\Big] dx,&&
		\end{align*}
		where the last term is non-negative since
		\begin{align*}
		&(2 \psi_1 \tilde{\psi}_1 + \tilde{\psi}_1^2)(2 \psi_2 \tilde{\psi}_2 + \tilde{\psi}_2^2) + \psi_1^2  \tilde{\psi}_2^2 + \psi_2^2 \tilde{\psi}_1^2 + \psi_1\psi_2 (\tilde{\psi}_1^2 + \tilde{\psi}_2^2) \\
		&= \psi_1 \psi_2 (\tilde{\psi}_1 + \tilde{\psi}_2)^2 + (\psi_1\tilde{\psi}_2 + \psi_2 \tilde{\psi}_1 + \tilde{\psi}_1\tilde{\psi}_2)^2 \geq 0.
		\end{align*}
Thus, for any $\tilde{\Psi} \in H_0^1(\R_+; \R^2)$, we have 
		$$
		G_\e (\Psi+\tilde{\Psi}) - G_\e (\Psi) \geq 0,
		$$
		with equality holding if and only if either $\tilde{\Psi} = 0$ or $\tilde{\Psi}=(\psi_2-\psi_1, \psi_1 - \psi_2)$. Therefore, the critical points $\Psi =(\psi_1,\psi_2)$ and $\Psi' =(\psi_2,\psi_1)$ are the only global minimizers of the energy $G_{\e}$ in $\mathcal{D}_{\alpha}$. 
	\end{proof1}

\section{Bifurcations of steady states}
\label{sec-bifurcations}

We start by studying bifurcations of steady states in the exceptional case $\gamma=1$. Bifurcations from the one-parameter family (\ref{rotating-state}) can be analyzed by the Lyapunov--Schmidt theory of bifurcations from a simple eigenvalue \cite{CR71}.

\vspace{0.25cm}

\begin{proof1}{\em of Theorem \ref{theorem-bifurcation}.}
	Let $\delta := \gamma - 1$ and rewrite the stationary Gross--Pitaevskii system (\ref{DWHPVR}) in the equivalent form
	\be\label{DWHPVR-rotated}
	\left.
	\begin{matrix} (-\e^2 \partial_x^2 + x^2 + \psi_1^2 + \psi_2^2 - 1) \psi_1 + \delta \psi_1 \psi_2^2 = 0, \\
		(-\e^2 \partial_x^2 + x^2 + \psi_1^2 + \psi_2^2 - 1) \psi_2 + \delta \psi_1^2 \psi_2 = 0,
	\end{matrix} \right\} \quad x \in \RR.
	\ee
	Substituting the decomposition
	\begin{equation}
	\label{decomposition-bif}
	\left.
	\begin{matrix}
	\psi_1(x) =  \cos \theta \; \eta_{\e}(x) + \varphi_1(x), \\
	\psi_2(x) =  \sin \theta \; \eta_{\e}(x) + \varphi_2(x)\end{matrix} \right\}
	\end{equation}
	into system (\ref{DWHPVR-rotated}) yields the perturbed system for $\varphi = (\varphi_1,\varphi_2)^T$ in the following form:
	\begin{equation}
	\label{perturbed-system-1}
	L_{\theta,\delta} \varphi +
	N_{\theta,\delta}(\varphi) + H_{\theta,\delta} = 0,
	\end{equation}
	where 
	\begin{eqnarray*}
		L_{\theta,\delta} & := & \left[ \begin{matrix}
			-\e^2 \partial_x^2 + x^2 - 1 + (1 + 2 \cos^2 \theta) \eta_{\e}^2 &
			2 \sin \theta \cos \theta \eta_{\e}^2 \\
			2 \sin \theta \cos \theta \eta_{\e}^2 & -\e^2 \partial_x^2 + x^2 - 1 + (1 + 2 \sin^2 \theta) \eta_{\e}^2
		\end{matrix} \right] \\
		& \phantom{t} & + \delta \left[ \begin{matrix}
			\sin^2 \theta \eta_{\e}^2 & 2 \sin \theta \cos \theta \eta_{\e}^2 \\
			2 \sin \theta \cos \theta \eta_{\e}^2 & \cos^2 \theta \eta_{\e}^2
		\end{matrix} \right],
	\end{eqnarray*}
	\begin{eqnarray*}
		N_{\theta,\delta}(\varphi) & := & \eta_{\e} \left[ \begin{matrix} 3 \cos \theta \varphi_1^2
			+ 2 (1+\delta) \sin \theta \varphi_1 \varphi_2 + (1 + \delta) \cos \theta \varphi_2^2 \\
			(1 + \delta) \sin \theta \varphi_1^2
			+ 2 (1+\delta) \cos \theta \varphi_1 \varphi_2 +  3 \sin \theta \varphi_2^2 \end{matrix} \right] \\
		& \phantom{t} & + \left[ \begin{matrix} (\varphi_1^2 + (1+\delta) \varphi_2^2) \varphi_1 \\
			((1+\delta) \varphi_1^2 + \varphi_2^2) \varphi_2  \end{matrix} \right],
	\end{eqnarray*}
and
	\begin{equation*}
H_{\theta,\delta} := \delta \sin \theta \cos \theta \eta_{\e}^3 \left[ \begin{matrix} \sin \theta \\
\cos \theta \end{matrix} \right],
\end{equation*}
	Due to the rotational invariance at $\delta = 0$, zero is a simple eigenvalue of $L_{\theta,\delta = 0}$ for every $\theta \in [0,2\pi]$ with the exact eigenvector
	\begin{equation}
	\label{zero-mode}
	L_{\theta,\delta = 0} \psi_0 = 0, \quad \mbox{\rm where} \;\; \psi_0 := \left[ \begin{matrix} -\sin \theta \\ \cos \theta \end{matrix} \right] \eta_{\e},
	\end{equation}
	which is related to the derivative of the ground state family (\ref{rotating-state}) in $\theta$. 
	This neutral mode inspires the Lyapunov--Schmidt decomposition in the form
	\begin{equation}\label{LS}
	\varphi = a \psi_0 + \hat{\varphi},
	\end{equation}
	subject to the orthogonality condition $\langle \psi_0, \hat{\varphi} \rangle = 0$,
	where $a \in \mathbb{R}$ is a parameter. 
	By the  Lyapunov--Schmidt bifurcation theory \cite{CR71}, there is a solution for $(a,\tilde{\varphi})$ for small $\delta \neq 0$ if and only if $\theta$ is a root of the bifurcation equation
	\begin{equation}
	\label{LS-constraint}
	\langle \psi_0, H_{\e,\delta} \rangle = \frac{1}{4} \delta  \sin 4\theta \| \eta_{\e} \|_{L^4}^4 = 0.
	\end{equation}
	The family of ground states (\ref{rotating-state}) is positive if $\theta \in [0,\frac{\pi}{2}]$. Simple zeros of the constraint (\ref{LS-constraint}) in $[0,\frac{\pi}{2}]$ occur for $\theta = 0$, $\theta = \frac{\pi}{4}$, and $\theta = \frac{\pi}{2}$. 
	The first and third roots correspond to the uncoupled states (S1) given by  (\ref{uncoupled-state}) 
	and the second root corresponds to the symmetric state (S2) given by (\ref{positive-state}). By the  Lyapunov--Schmidt bifurcation theory, 
there exists a unique continuation from the simple roots of the constraint (\ref{LS-constraint}). Hence, the three solutions uniquely
continue with $(a,\hat{\varphi})$ being a function of $\delta \neq 0$. 
(In fact, the three solutions are available exactly.) 
No other branches may bifurcate from the family of ground states (\ref{rotating-state}) as $\delta \neq 0$ ($\gamma \neq 1$).
\end{proof1}

\vspace{0.25cm}

We continue by studying bifurcations of domain wall state (S3) from the symmetric state (S2) near the bifurcation point $(\e,\gamma) = (0,1)$ within the energy functional $I_{\mu,1}$ defined in (\ref{energy-gamma-reduced}). The following lemma establishes the $\Gamma$-convergence of the energy  $I_{\mu,\gamma}$ in (\ref{energy-gamma}) to $I_{\mu,1}$ as $\gamma\to 1^+$.

\begin{lemma}
Fix  $\mu > 0$. Let $\Theta_{\gamma}$ be a sequence in $H^1_{\rm loc}(\R,\mathbb{R}^2)$ for $\gamma>1$. Then, 
	\begin{itemize}
		\item[(i)] If $\gamma \to 1^+$ and $I_{\mu,\gamma}(\Theta_{\gamma}) < \infty$ uniformly in $\gamma$ then we have a subsequence (not relabelled) $\Theta_{\gamma} \rightharpoonup \Theta$ in $H^1_{\rm loc}(\I,\mathbb{R}^2)$ and $|\Theta (x)|^2 =1$ for $x \in \I=\left(-\frac{1}{\mu},\frac{1}{\mu} \right)$.
		\item[(ii)] \emph{($\Gamma-\liminf$ inequality)} If $\gamma \to 1^+$ and $\Theta_{\gamma} \to \Theta$ in $L^2_{\rm loc}(\I,\mathbb{R}^2)$, then
		\begin{equation}
		\liminf_{\gamma \to 1^+}  I_{\mu,\gamma}(\Theta_{\gamma}) \geq I_{\mu,1}(\Theta).
		\end{equation}
		\item[(iii)] \emph{($\Gamma-\limsup$ inequality)} If $\Theta \in H^1_{\rm loc}(\I,\R^2)$, then there is a sequence 
		
		\noindent $\Theta_{\gamma} \in H^1_{\rm loc}(\R,\mathbb{R}^2)$ such that as $\gamma \to 1^+$ we have $\Theta_{\gamma} \to \Theta$ in $L^2_{\rm loc}(\I,\mathbb{R}^2)$ and 
		\begin{equation}
		\limsup_{\gamma \to 1^+} I_{\mu,\gamma}(\Theta_\gamma) \leq I_{\mu,1}(\Theta).
		\end{equation}
	\end{itemize}
	\label{theorem-gamma}
\end{lemma}

\begin{proof} 
For fixed $\mu>0$, we take $\e = \mu \sqrt{\gamma-1}$ and define the energy functional $I_{\mu,\gamma}$ in the form \eqref{energy-gamma} for $\gamma > 1$. 
Using the properties of $\eta_\e$ we know that 
$$\eta_\e(\mu \cdot) = \eta_{\mu \sqrt{\gamma-1}}(\mu \cdot) \to \eta_0(\mu \cdot)$$ locally uniformly on $\R$ as $\gamma \to 1^+$. 

(i) Assume $I_{\mu,\gamma} (\Theta_\gamma) < \infty$ uniformly in $\gamma$, we want to show that $\Theta_\gamma \rightharpoonup \Theta$ weakly in $H^1_{\rm loc}(\I;\R^2)$. We first observe that if $I_{\mu,\gamma} (\Theta_\gamma) < \infty$ then for any $0<R<\frac{1}{\mu}$ we have $\eta_\e(\mu \cdot) > C_1(R)>0$ on the interval $(-R,R)$ for $\e$ small enough and hence
$$
\int_{-R}^R  |\Theta_\gamma'|^2
+ |\Theta_\gamma|^2  \, dy \leq C.
$$
Therefore, we have $\Theta_\gamma \rightharpoonup \Theta$ weakly in $H^1_{\rm loc}(\I; \R^2)$ (up to a subsequence). Moreover, we also have 
$$
\int_{-R}^R (|\Theta_\gamma|^2 - 1)^2 \, dy\leq C (\gamma -1),
$$
where $C>0$ depends on $R$ but is independent of $\gamma$. Therefore, taking a limit in $\gamma$ we obtain $\int_{-R}^R (|\Theta|^2 - 1)^2 \, dy=0$ for any $0<R<\frac{1}{\mu}$ and hence $|\Theta(y)|^2 = 1$ for all $y \in \I$.

\vskip 0.2cm

(ii)  Let $\Theta_\gamma \in H^1_{\rm loc}(\R;\R^2)$ and $\Theta_\gamma \to \Theta$ strongly in $L^2_{\rm loc}(\I;\R^2)$. If $\liminf\limits_{\gamma \to 1^+} I_{\mu, \gamma} (\Theta_\gamma) = \infty$ there is nothing to prove, therefore we can assume $\liminf\limits_{\gamma \to 1^+} I_{\mu, \gamma} (\Theta_\gamma) <\infty$. Taking a subsequence $\Theta_\gamma$ (not relabelled) such that 
$$
\lim_{\gamma \to 1^+} I_{\mu, \gamma} (\Theta_\gamma) = \liminf_{\gamma \to 1^+}  I_{\mu, \gamma} (\Theta_\gamma) <\infty,
$$
we know that there is $\tilde\gamma>1$ such that for all $1<\gamma <\tilde\gamma$ we have $I_{\mu, \gamma} (\Theta_\gamma) \leq C < \infty$ with $C>0$ independent of $\gamma$. Therefore, using part (i), there is a subsequence (not relabelled) such that $\Theta_\gamma \rightharpoonup \Theta$ in $H^1_{\rm loc}(\I;\R^2)$ and $|\Theta(x)|^2=1$ for $x \in \I$. Fixing $0<R<\frac{1}{\mu}$ and using non-negativity of terms in $I_{\mu, \gamma}$ together with the properties of the $H^1$-norm and local uniform convergence $\eta_\e (\mu \cdot) \to \eta_0 (\mu \cdot)$ on $[-R,R]$ we have
\begin{align*}
 \liminf_{\gamma \to 1^+}  I_{\mu, \gamma} (\Theta_\gamma) &\geq  \liminf_{\gamma \to 1^+}  I_{\mu, \gamma} (\Theta_\gamma) \Big|_{[-R,R]} \\
 &\geq \frac12 \int_{[-R,R]} \left( \eta_0^2 (\mu \cdot) |\Theta'|^2 + \eta_0^4 (\mu \cdot)  \theta_{1}^2 \theta_{2}^2 \right) \, dy.
\end{align*}
Taking the limit as $R \to \frac{1}{\mu}$ and using the Monotone Convergence Theorem we obtain $\liminf\limits_{\gamma \to 1^+} I_{\mu, \gamma} (\Theta_\gamma) \geq I_{\mu,1} (\Theta)$. \\

(iii) Let $\Theta \in H^1_{\rm loc}(\I;\R^2)$, due to Sobolev embeddings we have $\Theta \in C_{\rm loc}(\I; \R^2)$. If $I_{\mu, 1} (\Theta) =\infty$ then the statement is trivially satisfied by taking 
$$
\Theta_\gamma(x) = 
\begin{cases}
\Theta(x) & \hbox{ if } x \in \left[ -\frac{1}{\gamma \mu}, \frac{1}{\gamma \mu} \right],  \\
\Theta( \frac{1}{\gamma \mu}) & \hbox{ if } x \notin 	\left[ -\frac{1}{\gamma \mu}, \frac{1}{\gamma \mu} \right], 
\end{cases}
$$
It is clear that $\Theta_\gamma \in H^1_{\rm loc} (\R; \R^2)$ and $\Theta_\gamma \to \Theta$ in $H^1_{\rm loc} (\I; \R^2)$.
Therefore we assume that $I_{\mu, 1} (\Theta) <\infty$. In this case we know that $|\Theta(x)|^2 =1$ for $x \in \I$ and we define
$$
\Theta_\gamma(x) = 
\begin{cases}
\Theta(x) & \hbox{ if } x \in \left[ -\frac{1-\e^{1/3}}{\mu}, \frac{1-\e^{1/3}}{\mu}\right], \\
\Theta( \frac{1-\e^{1/3}}{\mu}) & \hbox{ if } x \notin  \left[ -\frac{1-\e^{1/3}}{\mu}, \frac{1-\e^{1/3}}{\mu}\right], 
\end{cases}
$$
to obtain $\Theta_\gamma \to \Theta$ in  $H^1_{\rm loc}(\I;\R^2)$ (recall that $\e=\mu \sqrt{\gamma-1}$).  We want to prove that $\limsup\limits_{\gamma \to 1^+} I_{\mu, \gamma} (\Theta_\gamma) \leq I_{\mu,1}(\Theta)$. We define $\I_\e= (-\frac{1-\e^{1/3}}{\mu}, \frac{1-\e^{1/3}}{\mu})$ to obtain
\begin{align*}
 I_{\mu,\gamma} (\Theta_\gamma) &=  \frac{1}{2} \int_{\mathbb{R}} \left[ \eta_{\e}(\mu \cdot)^2 |\Theta_\gamma'|^2 
+   \eta_{\e}(\mu \cdot)^4 \theta_{\gamma,1}^2 \theta_{\gamma, 2}^2 \right] \, dy \\
&=  \frac{1}{2} \int_{\I_\e} \left[ \eta_{\e}(\mu \cdot)^2 |\Theta_\gamma'|^2
+   \eta_{\e}(\mu \cdot)^4 \theta_{\gamma,1}^2 \theta_{\gamma,2}^2 \right] \, dy +  \frac{1}{2} \int_{\mathbb{R} \setminus \I_\e}  \eta_{\e}(\mu \cdot)^4 \theta_{\gamma,1}^2 \theta_{\gamma,2}^2 dy \\
&\leq  \frac{1}{2} \int_{\I_\e} \left[ \eta_{\e}(\mu \cdot)^2 (|\theta_1'|^2 + |\theta_2'|^2) 
+   \eta_{\e}(\mu \cdot)^4 \theta_1^2 \theta_2^2 \right] \, dy + C \e^{1/3},
\end{align*}
where in the last inequality we used $|\Theta_\gamma|^2 =1$ and 
the estimate (\ref{P-3}) together with the fast decay of $\eta_{\e}(\mu y)$ to zero at infinity. Using (\ref{P-1-plus}), it is clear that
\begin{align*}
 I_{\mu,\gamma} (\Theta) &\leq  \frac{1}{2} \int_{\I_\e} \left[ \eta_{0}(\mu \cdot)^2 (|\theta_1'|^2 + |\theta_2'|^2) 
+   \eta_{0}(\mu \cdot)^4 \theta_1^2 \theta_2^2 \right] \, dy + C \e^{1/3}.
\end{align*}
Taking a $\limsup$ from both parts as $\gamma \to 1^-$ and using the  Monotone Convergence Theorem we obtain $\limsup\limits_{\gamma \to 1^+} I_{\mu, \gamma} (\Theta_\gamma) \leq I_{\mu,1}(\Theta)$.
\end{proof}

\begin{remark}
The above proof holds for $\mu>0$. In order to extend the result for $\mu = 0$  we can fix $\gamma>1$ and take the limit $\mu \to 0$. By Theorem~\ref{gconv1}, we have  $\Gamma$-convergence $I_{\mu,\gamma} \xrightarrow{\Gamma} I_{0,\gamma}$ as $\mu \to 0$ with
$$
I_{0,\gamma} (\Theta) =  \frac{1}{2} \int_{\mathbb{R}} \left[ 
(\theta_1')^2 + (\theta_2')^2
+  \frac{1}{2(\gamma-1)}  (\theta_1^2 + \theta_2^2  - 1)^2 +  \theta_1^2 \theta_2^2 \right] \, dy.
$$
By the same proof as in Lemma \ref{theorem-gamma}, we have $\Gamma$-convergence $I_{0,\gamma} \xrightarrow{\Gamma} I_{0,1}$ as $\gamma \to 1^+$, where the limiting energy functional $I_{0,1}$ is given by 
(\ref{energy-gamma-limiting}). 
\end{remark}

We can now give the proof of Theorem \ref{theorem-asymptotic}.

\vspace{0.25cm}

\begin{proof1}{\em of Theorem \ref{theorem-asymptotic}.}
We consider minimizers of $I_{\mu,1}$ 
given by (\ref{energy-gamma-reduced}) in the energy space $\mathcal{E}_s$ with the symmetry:
\begin{equation}
\label{symmetry-theta}
\theta_1(y) = \theta_2(-y), \quad y \in \RR.
\end{equation}
The symmetric state (S2) corresponds to the choice $\theta_1 = \theta_2 = \frac{1}{\sqrt{2}}$. The domain wall state (S3) corresponds to the choice 
with $\theta_1 \not\equiv \theta_2$. Perturbations $\tilde{\Theta} = (\tilde{\theta}_1,\tilde{\theta}_2)$ to these steady 
states in $\mathcal{E}_s$ satisfy the same symmetry (\ref{symmetry-theta}), hence they can be considered on the half-line $\RR_+$ subject to the Dirichlet conditions $\tilde{\theta}_1(0) - \tilde{\theta}_2(0) = 0$. 

Using the representation $(\theta_1,\theta_2) = (\sin(u),\cos(u))$
on the circle $\theta_1^2 + \theta_2^2 = 1$ and
writing $I_{\mu,1}(u)$ instead of $I_{\mu,1}(\Theta)$, we obtain
\begin{equation}
\label{energy-I}
I_{\mu,1}(u) = \frac{1}{2} \int_{\mathcal{I}_{\mu}} \left[ \eta_0(\mu \cdot)^2 (u')^2 + \frac{1}{4} \eta_0(\mu \cdot)^4 \sin^2(2u) \right] dy.
\end{equation}
The Euler--Lagrange equation  is 
\begin{equation}
\label{EL-I}
-\frac{d}{dy} \left[ \eta_0(\mu y)^2 \frac{du}{dy} \right] + \frac{1}{4} \eta_0(\mu y)^4 \sin(4u) = 0, \quad y \in \mathcal{I}_{\mu},
\end{equation}
with $u = \frac{\pi}{4}$ being the constant solution representing the symmetric 
state (S2). With the symmetry in $\mathcal{E}_s$, the second variation of $I_{\mu,1}(u)$ at $u = \frac{\pi}{4}$ in the direction $\tilde{u}$ is given by 
\begin{equation}
\label{second-var}
\delta^2 I_{\mu,1} = \int_0^{\mu^{-1}} \left[ \eta_0(\mu \cdot)^2 (\tilde{u}')^2 - \eta_0(\mu \cdot)^4 \tilde{u}^2 \right] dy
\end{equation}
subject to the Dirichlet condition $\tilde{u}(0) = 0$. We can rescale the domain as $x= \mu y$ and change $\tilde u(x) \to v(y)$ to obtain new expression for the second variation
\begin{equation}
\label{second-var1}
\delta^2 I_{\mu,1}  =\frac{1}{\mu} \int_0^{1} \left[ \mu^2\eta_0^2 (v')^2 - \eta_0^4 v^2 \right] dx.
\end{equation}
Since $\eta_0(x)^2 = (1 - x^2) {\bf 1}_{|x| \leq 1}$, the second variation is studied from the spectral problem with the spectral parameter $\nu$,
\begin{equation}
\label{Schr-spec}
- \frac{d}{dx} \left[ (1-x^2) \frac{dv}{dx} \right] = \nu (1-x^2)^2 v(x), \quad 0 < x < 1,
\end{equation}
subject to the Dirichlet condition $v(0) = 0$ and suitable boundary conditions 
at the regular singular point $x = 1$ such that the quadratic form (\ref{second-var1}) is finite. The index of the Frobenius theory at $x = 1$ is 
the double zero and the second solution $v(x) \sim \log(1-x)$ violates boundedness of the quadratic form (\ref{second-var1}). Therefore, 
the only relevant solution to \eqref{Schr-spec} satisfies $|v(1)| < \infty$, $|v'(1)|<\infty$ and therefore integrating  \eqref{Schr-spec} on $(x,1)$ we have $v'(1)=\lim_{x\to 1} v'(x)=0$. The spectral 
problem (\ref{Schr-spec}) with two boundary conditions has purely discrete positive spectrum with the smallest eigenvalue $\nu_0>0$ given by 
the Rayleigh quotient
$$
\nu_0 = \inf_{v \in H^1_0(0,1)} \frac{\int_0^1 (1-x^2) (v')^2 dx}{\int_0^1(1-x^2)^2 v^2 dx},
$$
where $H^1_0(0,1)$ refers to functions satisfying the Dirichlet condition 
at $x = 0$ but not at $x = 1$. Since
$$
\delta^2 I_{\mu,1}(v)\geq\frac{1}{\mu} (\mu^2 \nu_0 -1) \int_0^{1} \eta_0^4 v^2 \, dx,
$$
the constant solution $u = \frac{\pi}{4}$ is a local minimizer of $I_{\mu,1}$ for $\mu^2 > \nu_0^{-1}$ and a saddle point of $I_{\mu,1}$ for $\mu^2 < \nu_0^{-1}$. This gives the assertion of the theorem for  $(\theta_1,\theta_2) = (\sin(u),\cos(u))$ with $\mu_0 := \nu_0^{-1/2}$ if we can show that the constant solution $u = \frac{\pi}{2}$ is a global minimizer of $I_{\mu,1}$ when $\mu^2 \geq \nu_0^{-1}$. This is shown from the representation for any $\tilde u$:
\begin{align}
I_{\mu,1}\left(\frac{\pi}{4} +\tilde u\right) - I_{\mu,1}\left(\frac{\pi}{4}\right) &=   \int_0^{\mu^{-1}} \left[ \eta_0(\mu \cdot)^2 (\tilde u')^2 - \frac{1}{4} \eta_0(\mu \cdot)^4 \sin^2(2\tilde u) \right] dy \\ 
& \geq \int_0^{\mu^{-1}} \left[ \eta_0(\mu \cdot)^2 (\tilde u')^2 - \eta_0(\mu \cdot)^4 \tilde u^2 \right] dy \geq 0.
\end{align}
Hence, $u = \frac{\pi}{4}$ is a global minimizer of $I_{\mu,1}$ in 
the space with symmetry (\ref{symmetry-theta}) if $\mu \in [\mu_0,\infty)$
and is a saddle point if $\mu \in (0,\mu_0)$. Since $I_{\mu,1}$ admits a minimizer for every $\mu > 0$, the nonconstant state $u$ with $u(0) = \frac{\pi}{4}$ and $u'(0) \neq 0$ is a minimizer of $I_{\mu,1}$ for $\mu \in (0,\mu_0)$. 

By Lemma \ref{lemma-gamma-uniqueness} below, 
the constant state $u = \frac{\pi}{4}$ is the only solution 
of the Euler--Lagrange equaiton (\ref{EL-I}) with $u(0) = \frac{\pi}{4}$ for $\mu \in [\mu_0,\infty)$, 
therefore the nonconstant state does not exist for $\mu \in [\mu_0,\infty)$. 
\end{proof1}

The following lemma give uniqueness of the constant state $u = \frac{\pi}{4}$ among solutions of the Euler--Lagrange equation (\ref{EL-I}) satisfying 
$u(0) = \frac{\pi}{4}$ for $\mu \in [\mu_0,\infty)$.

\begin{lemma}
	\label{lemma-gamma-uniqueness}
	For $\mu \in [\mu_0,\infty)$ there is only one solution of 
	the Euler--Lagrange equation \eqref{EL-I} with $u(0) = \frac{\pi}{4}$ given by the constant state $u = \frac{\pi}{4}$.
\end{lemma}

\begin{proof}
	Let $u$ be a solution of the Euler--Lagrange equation \eqref{EL-I} and $I_{\mu,1}(u)$ be given by (\ref{energy-I}) and
	$v=(v_1,v_2)$. A direct calculation shows that for any admissible $u$ and $v$, we have
	\begin{align*}
	\frac{d^2}{dt^2} I_{\mu,1}(u+tv) &= \int_{\mathcal{I}_{\mu}} \left[ \eta_0(\mu \cdot)^2 (v')^2 + \eta_0(\mu \cdot)^4 \cos(4(u + t v)) v^2 \right] dy \\
	&\geq \int_{\mathcal{I}_{\mu}} \left[ \eta_0(\mu \cdot)^2 (v')^2 - \eta_0(\mu \cdot)^4  v^2 \right] dy .
	\end{align*}
	which is non-negative definite if $\mu \in [\mu_0,\infty)$. In particular, for any $u_1 \neq u_2,$ the function $f_{u_1,u_2}(t):=E(u_1+t (u_2-u_1))$ satisfies $(f_{u_1,u_2})''(t)>0,$ for $t\in (0,1).$
	
	Now, let $u_1=\frac{\pi}{4}$ and let $u_2 \not\equiv \frac{\pi}{4}$ be another solution of the Euler--Lagrange equation \eqref{EL-I} with $u(0) = \frac{\pi}{4}$. Since $u_1$ is a critical point of $I_{\mu,1}$, we have that $(f_{u_1,u_2})' (0)=0.$ On the other hand, since $u_1+t (u_2-u_1) \not\equiv \frac14(\pi +2\pi n)$ with $n \in \mathbb Z$, we have $(f_{u_1,u_2})''(t)>0$ for $t\in (\alpha,\beta) \subset (0,1)$. This implies that
	$$ 
	I_{\mu,1} (u_1)-I_{\mu,1}(u_2)= \int_{0}^{1} (f_{u_1,u_2})' (t) \, dt>0.
	$$
	We can reverse the roles of $u_1$ and $u_2$ to conclude $I_{\mu,1} (u_1)-I_{\mu,1}(u_2)<0$, which is a contradiction. So there is only one solution $u=\frac{\pi}{4}$ if $\mu \in [\mu_0,\infty)$.
\end{proof}

The final lemma clarifies the local (pitchfork) bifurcation 
among minimizers of $I_{\mu,1}$ at $\mu = \mu_0$. 
The nonconstant states bifurcate to $\mu < \mu_0$, where they are
global minimizers of $I_{\mu,1}$ by Theorem \ref{theorem-asymptotic}.

\begin{lemma}
	\label{lemma-normal-form}
The nonconstant state of the Euler--Lagrange equations 
(\ref{EL-I}) with $u(0) = \frac{\pi}{4}$ and $u'(0) \neq 0$ 
bifurcates to $\mu < \mu_0$ and satisfies 
$\| u - \frac{\pi}{4} \|_{L^{\infty}} \leq C \sqrt{\mu_0 - \mu}$
for some $\mu$-independent constant $C > 0$.
\end{lemma}

\begin{proof}
	We rescale again $x = \mu y$ and write $u(y) = \frac{\pi}{4} + v(x)$. The stationary equation (\ref{EL-I}) is rewritten in the form of the boundary-value problem 
\begin{equation}
\label{EL-I-normal}
\left. \begin{array}{l}
-\mu^2 \frac{d}{dx} \left[ (1-x^2)\frac{dv}{dx} \right] - \frac{1}{4} (1-x^2)^2 \sin(4v) = 0, \quad 0 < x < 1, \\
v(0) = 0, \quad v'(1) = 0. \end{array} \right\}
\end{equation}
The computation of the normal form for the local (pitchfork) bifurcation relies 
on the standard method of Lyapunov--Schmidt reduction \cite{CR71} as the zero eigenvalue of the linearized operator at $\mu = \mu_0$ is simple. 

Let $v_0$ be the eigenfunction of the spectral problem (\ref{Schr-spec}) for the smallest eigenvalue $\nu_0 > 0$, where $\mu_0 := \nu_0^{-1/2}$. 
To compute the coefficients of the normal form, we write 
$v = a v_0 + w$ and $\mu^2 = \mu_0^2 + \delta$, where $a$ and $\delta$ are parameters of the Lyapunov--Schmidt decomposition and $w$ is assumed to satisfy the orthogonality condition 
\begin{equation}
\label{pert-0}
\int_0^1 (1-x^2)^2 v_0(x) w(x) dx = 0.
\end{equation}
Decomposing (\ref{EL-I-normal}) into the projection to $v_0$ and to $w$ yields a system of two equations of the Lyapunov--Schmidt reduction:
\begin{equation}
\label{pert-1}
-\mu_0^2 \frac{d}{dx} \left[ (1-x^2)\frac{dw}{dx} \right] - (1-x^2)^2 \left[
w + F(w;a,\delta) \right] = 0
\end{equation}	
and
\begin{equation}
\label{pert-2}
\delta \int_0^1 (1-x^2)^2 v_0^2 dx = \frac{\mu_0^2}{4a} \int_0^1 (1-x^2)^2 v_0 \left[ \sin(4av_0 + 4 w) - 4 a v_0 \right] dx,
\end{equation}	
where 
$$
F(w;a,\delta) := \frac{\mu_0^2}{4 (\mu_0^2 + \delta)} \sin(4av_0 + 4 w) - a v_0 - w 
$$
and $w(x)$ satisfies the boundary conditions $w(0) = 0$ and $w'(1) = 0$. 
Since $F(0;a,\delta) = \mathcal{O}(|a|^3 + |a \delta|)$ as $(a,\delta) \to (0,0)$, there exists a unique solution $w$ to the nonlinear equation (\ref{pert-1}) under the conditions (\ref{pert-0}) and (\ref{pert-2}), 
for every small $(a,\delta) \in \mathbb{R}^2$, which satisfies the estimate
\begin{equation}
\label{est-w}
\| \sqrt{1-x^2} w' \|_{L^2} + \| (1-x^2) w \|_{L^2} \leq C(a^2 + |\delta|)|a|
\end{equation}
for some $(a,\delta)$-independent constant $C > 0$. Moreover, 
$w$ is in the domain of the linear operator associated with the spectral 
problem (\ref{Schr-spec}) satisfying the boundary conditions 
$w(0) = 0$ and $w'(1) = 0$ so that the estimate (\ref{est-w}) 
also extends to the $L^{\infty}(0,1)$ norm of the solution $w$. 

Substituting the solution $w$ of equation (\ref{pert-1}) into equation (\ref{pert-2}) and using \eqref{pert-0} defines uniquely $\delta$ for every small $a \in \mathbb{R}$ with the estimate 
\begin{equation}
\label{est-delta}
|\delta| \leq C_0 a^2 
\end{equation}
for some $a$-independent constant $C_0 > 0$. Moreover, the leading-order part in the expansion $\delta = \delta_2 a^2 + \mathcal{O}(a^4)$ can be found explicitly $$
\delta_2 = -\frac{8}{3} \mu_0^2 \frac{\int_0^1 (1-x^2) v_0^4 \, dx}{\int_0^1 (1-x^2) v_0^2 \, dx}.
$$
Since $\delta_2 < 0$, the nonconstant solution bifurcates for $\mu < \mu_0$ and satisfies 
$$
|\mu^2 - \mu_0^2| \leq C_0 a^2, \qquad 
\| v - a v_0 \|_{L^{\infty}} \leq C_0 |a|^3
$$ 
for some $a$-independent constant $C > 0$. 
Hence $a = \mathcal{O}(\sqrt{\mu_0 - \mu})$ and the solution satisfies 
$\| u - \frac{\pi}{4} \|_{L^{\infty}} \leq C \sqrt{\mu_0 - \mu}$
for some $\mu$-independent constant $C > 0$.
\end{proof}

\begin{remark}
	The sign of $\delta_2$ in the normal form in Lemma \ref{lemma-normal-form} suggests that the domain wall states are global minimizers of the energy $G_{\e}$ in $\mathcal{E}_{s}$ in the region of their existence at least for small $\e > 0$. The proof of non-existence of the domain wall states for $\gamma \in (1,\gamma_0(\e))$ for general $\e \in (0,\e_0)$ is left open for further studies.
\end{remark}

\end{document}